\documentclass[12pt]{article}
\usepackage{amsmath,amsfonts,amsthm, amssymb}
\usepackage{subeqnarray,pdfsync}
\usepackage{pdfsync}
\usepackage{graphicx, epstopdf}
\usepackage{stmaryrd}
\usepackage{hyperref}
\usepackage{subeqnarray, cases}
\def\be{\beta}
\def\ga{\gamma}
\def\de{\delta}

\def\Om{\Omega}    \def\om{\omega}

\newtheorem{claim}{Claim}

\newtheorem{theorem}{Theorem}
\newtheorem{maintheorem}{Theorem}

\newtheorem{corollary}[theorem]{Corollary}
\newtheorem{definition}[theorem]{Definition}
\newtheorem{example}[theorem]{Example}
\newtheorem{lemma}[theorem]{Lemma}

\newtheorem{proposition}[theorem]{Proposition}
\newtheorem{remark}[theorem]{Remark}

\theoremstyle{definition}

\usepackage{color, fancybox, empheq}

\newcommand{\comment}[1]{}

\newcommand\JM{\color{blue}}

\newcommand\RL{\color{red}}

\newdimen\AAdi%
\newbox\AAbo%
\def\AArm{\fam0 }%\tenrm}%
\def\AAk#1#2{\setbox\AAbo=\hbox{#2}\AAdi=\wd\AAbo\kern#1\AAdi{}}%
\def\AAr#1#2#3{\setbox\AAbo=\hbox{#2}\AAdi=\ht\AAbo\raise#1\AAdi\hbox{#3}}%
\def\BBone{{\AArm 1\AAk{-.8}{I}I}}%

\newcommand {\CA}{{\mathcal A}}

\newcommand {\CC}{{\mathcal C}}

\newcommand {\CM}{{\mathcal M}}

\newcommand{\disp}{\displaystyle}
\newcommand{\eps}{\varepsilon}
\newcommand{\8}{\infty}

\def\m1{{-1}}

\newcommand{\ol}{\overline}
\def\S{\Sigma}
\def\s{\sigma}
\def\l{\lambda}
\def\supp{\mbox{supp}\,}

\newcommand{\wt}{\widetilde}

\def\ie{{\em i.e.,\ }}

\def\ga{\gamma}

\newcommand{\R}{\mathbb{R}}

\title{On the selection of subaction and measure for perturbed potentials}

\author{R. Leplaideur\thanks{ISEA \& LMBA UMR6205, Université de la Nouvelle Calédonie} \,\, and\,\, J. K. Mengue\thanks{Departamento Interdisciplinar, Universidade Federal do Rio Grande do Sul} }

\begin{document}

\maketitle

\begin{abstract}
 We prove that when the Aubry set for a Lipschitz continuous potential is a subshift of finite type, then the pressure function converges exponentially fast to its asymptote as the temperature goes to 0. The speed of convergence turns out to be the unique eigenvalue for the matrix whose entries are the \emph{costs} between the different irreducible pieces of the Aubry set. \\
 For a special case of Walter potential we show that pertubation of that potential that go faster to zero than the pressure do not change the selection, nor for the subaction, neither for the limit measure a zero temperature.
\end{abstract}



%%%%%%%%%%%%
%%%%%%%%%%%%

\section{Introduction}
%%%%%
\subsection{Background}

In this paper we study the question of \emph{selection in ergodic optimization}. Given a dynamical system $(X,T)$, ergodic theory describes \emph{almost all} orbits, with respect to some/any \emph{invariant probability}. The thermodynamic formalism is a way to select particular invariant measure(s): one fixes some potential $A:X\to \R$ and one considers the measure(s) maximizing  the free energy 
$$P(\be\cdot A):=\sup_{\mu\ T-inv}\left\{h_{\mu}(T)+\be\cdot\int A\,d\mu\right\},$$
where $h_{\mu}(T)$ is the Kolmogorov entropy and $\be$ is a real parameter, corresponding to the inverse of the temperature in statistical mechanics. Such a measure is called an equilibrium state (for $\be A$).
This formalism has been introduced in the Dynamical Systems settings in the 70's, mainly by Sinai, Ruelle and Bowen. 

Ergodic optimization is another way to select particular invariant measures. Instead of maximizing the free energy, one simply maximizes the integral of the potential: an invariant (probability) measure $\mu$ is $A$-maximizing if it satisfies{
$$m(A):=\sup_{\nu\ T-inv}\int A\,d\nu=\int A\,d\mu.$$}
Existence of equilibrium states usually follows from upper-semi-continuity for the entropy, whereas existence of maximizing measures simply follows from continuity on the compact convex set of invariant probabilities. 

Furthermore, there  is a relation between equilibrium states and maximizing measures: any accumulation point for the equilibrium state $\mu_{\be A}$ as $\be$ goes to $+\8$ is an $A$-maximizing measure. It is also known (see \cite{Contreras2016}) that if $(X,T)$ is uniformly hyperbolic, then generically in the Lipschitz topology for the potential, the family of equilibrium states $\mu_{\be A}$ converges to the unique $A$-maximizing measure which is supported on a periodic orbit. 

The question of \emph{selection in ergodic theory} deals with the residual case. In the case of existence of several maximizing measures, is there convergence for $\mu_{\be A}$ (as $\be\to+$) and/or what makes that some accumulation point\footnote{usually called a \emph{ground state} in statistical mechanics.} is selected instead of another one ? 

Despite residual sets may be considered as small, we emphasize that it is usually extremely easy to find $A$ which admits several maximizing measures: pick any 2 disjoints invariant compact subsets $K_{1}$ and $K_{2}$ and consider the potential $A:=-d(.\,, K_{1}\cup K_{2})$. 
The question of selection is thus meaningful. 

About selection, few is known. If $(X,T)$ is a subshift of finite type and if $A$ is locally constant, then convergence to a ground state as temperature goes to zero holds (see \cite{Bremont, Leplaideur}). { On the other hand, there are Lipschitz potentials where $\mu_{\beta A}$ does not converge (see \cite{Chazottes, BGT, CL}).}  In \cite{Leplaideur2} it is proved that {(under certain hyphothesis)} flatness of the potential closed to the \emph{Aubry set} determines which pieces may support ground states. { On the other hand, in \cite{BLM} it was proved that the knowledge of the potential in a neighborhood of the Aubry set is not sufficient to determine the result of selection.}
 In \cite{BLL-selc}, all the difficulty of the selection appeared: equilibrium states are usually obtained via an operator\footnote{called the Transfer or the Ruelle-Perron-Frobenius Operator.} and are constructed from its \emph{eigen-function} and  \emph{eigen-measure} (see below). Therein, it is shown that the eigen-measure selects one piece, whereas the eigen-function selects the other piece of the Auby set. Then, considering the value of the potential  ``in the middle'' as a parameter, the ground state may change as this parameter changes. A complete description of the limit depending on the parameter was given and the main point was that this selection process was strongly discontinuous. The main reason for that is that the phase diagram results from \emph{Max-Plus} formalism, explaining the discontinuities. 

Several properties and relations between equilibrium states and maximizing measures may be understood from convex analysis. The pressure function $\be\mapsto P(\be A)$ is convex, hence its slope is increasing. This is related to the fact that accumulation points for $\mu_{\be A}$ are $A$-maximizing. Furthermore, the graph of the pressure admits an asymptote as $\be\to+\8$ of the form $h_{max}+\be m(A)$, where {$h_{max}$ is the maximal entropy among $A$-maximizing measures.}
	%$m(A)$ is the maximal value for the integrals of $A$ and 	 $h_{max}$ turns out to be the maximal  entropy among $A$-maximizing measures.} 
Some of these properties are re-explained below. We mention \cite{BKL} for further descriptions of how thermodynamic formalism, including freezing phase transition (that is as $\be\to+\8$) is related to convex analysis. 

\bigskip
Our first result (see Theorem \ref{theoA}) shows that if the Aubry set is  a subshift of finite type, then the pressure converges exponentially fast to its asymptote. In \cite{CG} an upper-bound in $O(1/\be^{2})$ was proved. Moreover, we show that the exponential rate is the unique eigen-value of a transition matrix within the Max-Plus formalism, where entries are related to the Peierl's barrier of the potential between the different pieces of the Aubry set.

\bigskip
Our second result (see Theorem \ref{theoB}) studies, in the case of Walters potentials, the stability of the selection under small perturbation of the potential. In that case Theorem \ref{theoA} holds, and we show that if the perturbation goes faster to zero than how the pressure goes to its asymptote, then the selection does not change.

%\newpage

%%%%%
\subsection{Settings and results}
In the following  $X\subseteq \{0,1,...,d\}^{\mathbb{N}}$ is a subshift of finite type given by an aperiodic matrix. $X$ equipped with the metric $d$ defined by 
\[d((x_{0},x_1,x_2,x_3,...),(y_{0},y_1,y_2,y_3,...)) = \theta^{\min\{i\,|\,x_i\neq y_i\}},\,\theta \in (0,1)\]
is a compact space. 
The shift map $\sigma:X \to X$ is defined by
\[\sigma(x_{0},x_1,x_2,x_3,...) = (x_{1},x_2,x_3,x_4,...).\]

In the special case $d=2$, the full-2-shift is denoted by $X_{2}$. 
Points in $X$ are also called \emph{infinite words}, hence $X_2$ is the set of all infinite words with 0 and 1. 
If $a_{0},\ldots, a_{n}$ are \emph{digits or letters} in $\{0,1,\dots, d\}$, the cylinder $[a_{0},\ldots ,a_{n}]$ is the set of points  $ x\in X$ such that $x_{i}=a_{i}$ for $0\le i\le n$. 

\bigskip
Given a Lipschitz continuous function $A:X \to \mathbb{R}$, we denote by $|A|_\infty=\sup_{x\in X} |A(x)|$ the supremum norm of $A$ and by $Lip(A) = \sup_{x\neq y} \frac{|A(x)-A(y)|}{d(x,y)}$ the Lipschitz constant of $A$.  The Lipschitz norm is $|\,.\, |_{\8}+Lip(.)$.

The transfer operator $L_A$ is defined by
\[L_{A}(\psi)(x) = \sum_{\sigma(z) =x}e^{A(z)}\psi(z).\]
We refer to \cite{Bowen} for complete study of this operator. It acts on $\CC(X)$ and on the space  $\CC^{0+1}(X)$ of Lipschitz continuous functions. Its spectral radius $\l_{A}$ (for $|\quad |_{\8}$-norm) is a single dominating eigenvalue on $\CC^{0+1}(X)$. It also turns out to be equal to $e^{P(A)}$. The associated 1-dimensional  eigen-space is $span(H_{A})$, where $H_{A}$ is Lipschitz continuous, positive and uniquely determined up to some  normalization. The dual operator $L_{A}^{*}$ for the $|\quad |_{\8}$-norm acts on the set of measures and $\nu_{A}$ is the unique probability satisfying $L_{A}^{*}(\nu_{A})=\l_{A}\nu_{A}$. It is referred to as the \emph{eigenmeasure} or the \emph{conformal measure}. 

The usual normalization for $H_{A}$ is $\disp\int H_{A}\,d\nu_{A}=1$, but for Theorem \ref{theoB} we shall choose another normalization.

\bigskip
We now recall some notion on ergodic optimization. We refer to \cite{BLL} for proofs and more details. 
The set of $\s$-invariant probabilities is denoted by $\CM(\s)$. We set 
$$m(A):=\max_{\mu\in\CM(\s)}\int A\,d\mu.$$
Any measure realizing this supremum is called $A$-maximizing or maximizing for $A$. 
We denote by $\CM_{max}(A)$ the set of $A$-maximizing measures. 

A   \emph{calibrated subaction} for $A$ is a Lipschitz continuous function $V:X\to\R$ satisfying for every $x\in X$ 
$$\max_{\sigma(z)=x} [A(z) + V(z) - V(x) - m(A)] = 0.$$
It is also known (see \cite{BLL} and Subsubsection \ref{sssec-calisuba1}) that any accumulation point for $\disp\dfrac1\be\log H_{\be A}$ (and for $|\quad|_{\8}$) is such a calibrated subaction. 
Setting $g(z):=A(z)+V(z)-V\circ \s(z)-m(A)$, we see that $g$ is a Lipschitz continuous function, cohomoulogous to $A$ (up to an additive constant) and non-positive. 
It thus satisfies that $m(g)=0$, $P(\beta g) = P(\beta A)-\beta m(A)$ and the unique equilibrium sate for $\be A$ is also the unique equilibrium state for $\be g$. Furthermore a $A$-maximizing measure is $g$-maximizing and conversely. 

Hence, without loss of generality one shall {\bf assume} in this paper that $A$ is non-positive and satisfies $m(A)=0$ (except when this assumption is not required for general results on ergodic optimization). 

We remind that $S_{n}(A)$ stands for the Birkhoff sum $A+\ldots +A\circ \s^{n-1}$.

\begin{definition}
\label{def-maneaubrymather}
With these notations and assumptions, the \textbf{Ma\~n\'e potential} associated to $A$ is defined by {
	$$S(x,y):=\lim_{\epsilon \rightarrow 0}\left[\sup\{S_n(A)(z),\ \s^{n}(z)=y,\ d(x,z)<\eps\}\right],$$
the \textbf{Aubry set} of $A$ is defined by
	$$\Omega:=\{x\in X\, |\, S(x,x) = 0\}$$ and} \textbf{Mather set}  is 
	$$\mathcal{M} = \cup_{\mu \in \mathcal{M}_{\max}(A)}\operatorname{supp}(\mu).$$
\end{definition}
We remind that $y\mapsto S(x,y)$ is Lipschitz continuous, whereas $x\mapsto S(x,y)$ is only upper semi-continuous. 
By definition it is clear that $S(x,y)$ is non-positive.

\begin{remark}
\label{rem-aubry}
The {Mather} set is non-empty (as it contains {the support of} any accumulation point for $\mu_{\be A}$) and is contained in the Aubry set (see \cite[Th 3.15]{BLL}). This shows that the Aubry set is not empty either. 
%$\blacksquare$
\end{remark}

\bigskip
With these settings our first main result is: 
\begin{maintheorem}\label{theoA}
Let $A$ be a non-positive Lipschitz continuous  function satisfying $m(A)=0$. Suppose that the Aubry set $\Omega$ is a subshift of finite type with topological entropy $h$. 
 Then  the limit 
\begin{equation}\label{gamma}
\gamma :=\lim_{\beta\to\infty}\frac{1}{\beta}\log(P(\beta A)-h)
\end{equation}
 exists. \end{maintheorem} 
 
 {In the proof of above theorem we show that $\ga$ is the unique eigenvalue in max-plus formalism of a matrix which is constructed from an analysis of the {Ma\~n\'e potential} between the irreducible components of $\Om$ with maximal entropy (see Prop. \ref{gamma equal entropies} below).}

\bigskip

{ We can now deal with the question of selection. The question we are interested in is to know how/if selection of subaction and measure are stable under perturbations. 
%	If $B:X\to\R$ is Lipschitz continuous, it is known that any accumulation point  for $\mu_{\be A+B}$ as $\be\to+\8$ is a $A$-maximizing measure with maximal free-energy for $B$ among the set of $A$-maximizing measures. We give in Section \ref{Sec-toto} a proof of this fact. 
%This implies that our problem makes sense only when $B$ goes to 0 as $\be$ goes to $+\8$. 
As we will see it is easy to exhibit examples where the selections for the family $\be A+B$ differ of the selections for the family $\be A$.  
Therefore, our question is{\JM:} which conditions on a family of functions $B_{\be}$ do ensure that $\mu_{\be A+B_{\be}}$ selects the same ground state than $\mu_{\be A}$ and $\frac{1}{\beta}\log(H_{\be A +B_{\be}})$ selects the same subaction than $\frac{1}{\beta}\log(H_{\be A})$  ?}

It turns out that we can give an answer in a special case of potentials for $X_{2}$. { For $i\in \{0,1\}$, let $f_i:[0,1]\to (-\infty, 0]$ be a Lipshitz function satisfying $f_i^{-1}(0)=\{0\}$ and let $A:\S_2\to (-\infty,0]$ be the potential defined by
	\[A(x)=\left\{\begin{array}{c} f_0( d(x,0^{\infty})),\, \text{if} \,x\in [0] \\ f_1( d(x,1^{\infty})),\, \text{if}\, x\in [1]\end{array}\right..\]
These potentials belong to the family which were introduced by  P. Walters in \cite{Walters}.} They represent a class of functions for which complete computations are possible, {which allows to make conjecture for more complicated cases}. We emphasize that in that case, the Aubry set is the union of two fixed points $\{0^{\8}, 1^{\8}\}$. It is thus a subshift of finite type with topological entropy equal to 0.  Hence, Theorem \ref{theoA} holds in that case. 

{We emphasize that our result gives another proof for the convergence of the Gibbs measures associated to these potentials (see \cite{BGT}) as the temperature goes to zero, for the case where convergence holds. }

\begin{maintheorem}\label{theoB} 
Consider a Lipschitz function $A:X_2 \to \mathbb{R}$ satisfying 
{\begin{equation}\label{W}
A(0^{\infty})=A(1^{\infty})=0,\,\,	A|_{[01]}=b,\,\,A|_{[10]}=d,\,\,  A|_{[0^n1]} =a_n,\,\,A|_{[1^n0]} =c_n,   
	\end{equation}}
for negative numbers $b,d,\{c_n\}_{n\geq 2}$ and $\{a_n\}_{n\geq 2}$. Assume $H_{\be A}$ is normalized by $H_{\be A}(0^{\8})=1$. 
Then, 

\begin{enumerate}
\item $\ga:=\lim_{\be\to+\8}P(\be A)$ exists and it is negative;
\item For any $\de<\ga$, for any family $(B_{\be})$ of Lipschitz functions such that $|B_\beta|_\infty<e^{\beta \delta}$ we have $$\lim_{\beta \to \infty} \mu_{\beta A+B_\beta} =\lim_{\beta \to \infty} \mu_{\beta A};$$ 

\item For any $\delta<\gamma$ and any $c\in\mathbb{R}$, for any family $(B_\beta)$ such that $|B_\beta|_\infty<e^{\beta \delta}$ and $Lip(B_\beta)<\beta c$ we have $$\lim_{\beta \to \infty}\frac{1}{\beta}\log(H_{\beta A+B_\beta})=\lim_{\beta \to \infty}\frac{1}{\beta}\log(H_{\beta A}).$$

\end{enumerate}
\end{maintheorem}

	\begin{remark} We emphasize that the perturbations $B_{\be}$ do not need  to satisfy similar hypothesis as $A$ in \eqref{W}.
		\end{remark}

{ \begin{remark}
\label{rem-nega}
Assumptions may be released. We let the reader check that we actually only need $A(0^{\infty})=A(1^{\infty})=0$ and:
\begin{enumerate}
\item $A$ non-positiv,
\item $b+d<0$, 
\item $a_{2}<0$, 
\item $c_{2}<0$.
\end{enumerate}

$\blacksquare$

%{\JM To study if we could remove this remark. It is necessary at least to suppose also that the Mather set is $\{0^\infty,1^\infty\}$ and it is not so easy to ensure that the remark is correct (I believe it is true).
%	
%	For example, we use in computations for selection of subaction that $S(0^\infty, 10any\,thing)=\sum a_n + b$. When the numbers are negative this is easy to see, but for positive numbers is more complicated. For example, the sum
%	$S_{n+3}(0^n11010any\,thing) $ is $(\sum a_n)+b+c_2+d+b$	  
%	If $c_2$ is positive we need $b+d+c_2<0$ in order to have $S(0^\infty, 10any\,thing)=\sum a_n + b$ yet. Clear, this is really true if we suppose that the periodic orbit $(011)^\infty$ is not maximizing.
%	
%	Anyway even $A$ is just Walters then $H_A$ is Walters too and then $A$ admits subactions which are also Walters. Using the subactions as coboundary we will get a new function which is not positive (but it is just $\leq 0$ and not $<0$ out of Mather set). }
\end{remark}}

For that special case of {potentials satisfying \eqref{W}}, existence of $\gamma$ follows from Theorem \ref{theoA} but has also been already proved in \cite{BLM, BGT}. Therein, existence of the limits $\lim_{\beta \to \infty} \mu_{\beta A}$ and $\lim_{\beta \to \infty}\frac{1}{\beta}\log(H_{\beta A})$ were also proved. 
Here,  one novelty is that we give another characterization for the limit {$V:=\lim_{\beta \to \infty} \frac{1}{\beta}\log(H_{\beta A})$ in terms} of max-plus formalism and the main novelty is that we give condition to get stability.

%%%%%%
\subsection{Plan of the paper}

The paper is organized in the following way. 

In section \ref{sec: preliminary} we introduce a basic analysis of the problem of selection an some results that will be useful in the proof of Theorem \ref{theoA} and in the study of selection of subaction. 

In section \ref{secproofthA} we prove Theorem \ref{theoA}. The proof is divided in subsections in order to facilitate the reading. {Using the fact that a subaction is entirely determined by its values on the Aubry set, we obtain a relation between these values on the irreducible component of the Aubry set in the case it is a subshift of finite type. In this relation, the speed of convergence of the pressure turns out to be an eigenvalue for a special matrix.}

{In section \ref{sec: theoB} we prove Theorem \ref{theoB}. In the first subsection we use Theorem \ref{theoA} to identify $\ga$. Proposition \ref{prop-autoespaco} is the key point to get selection of the subaction.  In Proposition \ref{prop:limits measure} we show the convergences to limit measures for $\mu_{\beta A +B_\beta}$ and $\mu_{\beta A}$. 
In Section \ref{initial} we give a totally  computational example where $B_{\be}$ goes to zero slower than $e^{\be\cdot \ga}$ and then there is a change of selections. }

%%%%%%
\subsection{Motivations}

 An initial motivation for the study of selection for perturbed potentials was given from the results in \cite{LMMS}. In such paper very similar characterizations concerning zero temperature limits for $\mu_{\beta A}$ and $H_{\beta A}$ was getting  when supposing $\tilde{X}=M^{\mathbb{N}}$, where $M$ is a compact metric space instead a finite set. In this case the formulation of the Ruelle operator uses a prior probability measure $m$ satisfying $\operatorname{supp}(m)=M$ and then
$L_{A,m}:C(\tilde{X})\to C(\tilde{X})$ is given by $L_{A,m}(\psi)(x)=\int_M e^{A(ax)}\psi(ax)\,dm(a)$. 

A natural question concerning this general setting is the dependence of the zero temperature limits with respect to the prior probability measure $m$ which is used to define the operator. The answer is obtained by considering this question in the particular case where $M=\{0,1,...,d\}$ and $\tilde{X}=\{0,1,...,d\}^{\mathbb{N}}$. For any prior probability $m=(p_0,...,p_d)$, where $p_i>0\,\forall i\in\{0,1,...,d\}$, the operator $L_{\beta A,m}$ can be written as
\[L_{\beta A,m}(\psi)(x)=\sum_{a=0}^{d}e^{\beta A(ax)}\psi(ax)p_{a} = \sum_{a=0}^de^{\beta A(ax)+\log(p_a)}\psi(ax),\]
for each $\beta>0$. So $L_{\beta A,m}$ is just the usual Ruelle Operator associated to the perturbed potential $\beta A(x)+\log(p_{x_1})$, where $x=(x_1,x_2,x_3,...)$. It is easy to construct examples (see Example \ref{example:initial} below) where such perturbation will change the results of selection of subaction and measure. Therefore, the conclusion is that selection of subaction and measure in the setting of \cite{LMMS} depends of the prior measure $m$. In the case of finite alphabet $\{0,1,...,d\}$ a natural next step is to study what are the conditions on the perturbation in order to get equal limits.  

\bigskip
{ A second motivation is to inquire how perturbation could enlarge the class of known potentials for which there is convergence. As it is said above, it is known that convergence occurs as soon as the potential is locally constant. In that case the Aubry set is a subshift of finite type. It is thus quite natural to inquire if, for a given potential, one can find some sequence of locally constant perturbations,, refining the initial potential,  with all the same Aubry set, and such that convergence for this family finally yields convergence for the initial potential. This is a research program, and one first step is to control how perturbations influence the selection.}

%%%%%%%%%%%%
%%%%%%%%%%
\section{Auxiliary  and preliminary results}\label{sec: preliminary}

%%%%%
\subsection{Convexity for the pressure and consequences}
{ 
Here we do not necessarily assume that $m(A)=0$ holds.

\begin{definition}
\label{def-presmax}
For any continuous $B:X\to\mathbb{R}$ we set 
$$P_{A-max}(B):=\sup_{\mu \in  \mathcal{M}_{max}(A)} \left\{\int B\,d\mu+h_\mu(\sigma)\right\}.$$
Any measure that realizes the maximum is called a $A$-maximizing equilibrium state  for $B$.
\end{definition}

\begin{lemma}\label{pressureselect} 
	Let $A$ be a Lipschitz function and let $B_\beta:X\to\mathbb{R}$ be a family of  Lipschitz functions converging uniformly to $B$. 
	Then, any accumulation point for $\mu_{\be A+B_{\be}}$ as $\be\to+\8$ is $A$-maximizing equilibrium state for $B$. 
	
	Furthermore, if $B$ is Lipschitz continuous, then $\be\mapsto P(\beta A +B)-\beta m(A)$ is non-increasing and converges to $P_{A-max}(B)$.  
\end{lemma}
\begin{proof}
The proof below follows ideas present in \cite{CG}. Let $\mu$ be any maximizing measure to $A$	and let $\mu_\infty$ be an accumulation measure of $\mu_{\beta A +B_\beta}$. From
	\[\beta \int A\, d\mu + \int B_\beta\,d\mu + h(\mu) \leq  \beta\int A\, d\mu_{\beta A +B_\beta} + \int B_\beta\,d\mu_{\beta A +B_\beta} + h(\mu_{\beta A +B_\beta})\]
	we conclude that $\int A\, d\mu\leq \int A\, d\mu_\infty$. This proves that $\mu_\infty \in M_{max}(A)$.
	
	We have also
	\[\int B_{\beta}\, d\mu + h(\mu)\leq P(\beta A + B_\beta) - \beta\int A\, d\mu\]
	\[ = \beta(\int A\, d\mu_{\beta A +B_\beta}- \int A\,d\mu) + \int B_\beta\, d\mu_{\beta A +B_\beta} +h(\mu_{\beta A +B_\beta})\]
	\[\leq \int B_\beta\, d\mu_{\beta A +B_\beta} +h(\mu_{\beta A +B_\beta}).\]
	Using the upper semi-continuity of the entropy, as $\beta \to \infty$ we get
	\[\int B\, d\mu + h(\mu) \leq   \int B d\mu_{\infty} +h(\mu_\infty).\]
{This yields that $\mu_{\8}$ has maximal pressure among measures in $M_{max}(A)$. }

%\[P_{A-max}(B):=\sup_{\mu \in  \mathcal{M}_{max}(A)} \int B\,d\mu+h_\mu(\sigma) = \int B\,d\mu_\infty +h_{\mu_\infty}(\sigma).\]

\bigskip

{ Let us now assume  that $B$ is Lipschitz continuous. 
We remind that the pressure is defined by 
	$$P(A):=\max\{h_{\mu}(\s)+ \int A\,d\mu\},$$
	which immediately yields  that $\be\mapsto P(\be A+B)$ is convex. It follows from Prop. 4.10 in \cite{PP} that
	\[\frac{d}{d\beta}P(\beta A +B) = \int A\, d\mu_{\beta A+B}.\]  
	We can also observe that, by definition of pressure, for any $\be_{0}$ the line $\beta \to h_{\mu_{\be_{0}A+B}}(\s)+\be \int A\,d\mu_{\be_{0}A+B}$ is below the graph of the pressure function $\beta \to P(\beta A +B)$ and touches it at $\be=\be_{0}$. 
	
	As $$\frac{d}{d \beta}[P(\beta A +B) - \beta m(A)] = \int A\, d\mu_{\beta A+B} -m(A) \leq 0$$ we get that $\beta \to P(\beta A +B)-\beta m(A)$ is a non-increasing function. 
Let $\mu_\infty$ be any accumulation point for $\mu_{\beta A+B}$. As, by definition of pressure, 
	$$ P(\be A+B)-\be m(A)-\int B\,d\mu_\infty-h_{\mu_\infty}(\s)\geq 0$$ we obtain $P(\be A+B)-\be m(A)\geq \int B\,d\mu_\infty +h_{\mu_\infty}(\sigma) $.
	On the other hand, as the entropy is upper semi-continuous,
	\[\lim_{\beta\to +\infty}  P(\be A+B)-\be m(A) \leq \limsup_{\beta\to +\infty} \int B\,d\mu_{\beta A+B} +h_{\mu_{\beta A+B}}(\sigma) \leq  \int B\,d\mu_\infty +h_{\mu_\infty}(\sigma) .\]
}
\end{proof}

\subsubsection{Selection depends on $B$}
The next example shows that if $B_{\be}$ does not go to 0, then selection may (easily) change. { It is a consequence of Lemma \ref{pressureselect} .}
 
 \begin{example}\label{example:initial}
 Set $A$ any non-positive Lipschitz continuous function on $X_2$, negative everywhere except at $0^{\8}$ and $1^{\8}$. Pick $B_{0}$ and $B_{1}$ Lipschitz continuous functions such that $B_{0}(0^{\8})>B_{0}(1^{\8})$ and $B_{1}(0^{\8})<B_{1}(1^{\8})$. 
 \end{example}
{ It is easy to check that $\de_{0^{\8}}$ is the unique $A$-maximizing equilibrium state for $B_{0}$, whereas $\de_{1^{\8}}$ is the unique $A$-maximizing equilibrium state for $B_{1}$. Hence}, $\mu_{\be A+B_{0}}$ converges to $\de_{0^{\8}}$ whereas $\mu_{\be A+B_{1}}$ converges to $\de_{1^{\8}}$.

%%%%%%%%%%%%
\subsubsection{Change of selection even if $B_{\be}$ goes to 0}

\begin{corollary}
\label{coro-sele-zero}
Let $A:\{0,1\}^{\mathbb{N}}\to\mathbb{R}$ be a function which is equal to zero at the points $0^{\infty}$ and $1^{\infty}$ and negative at the other points. Suppose that $\mu_{\beta A} \to \delta_{1^{\infty}}$. Then there exists a family of potentials $B_{\be}$ { which uniformly converge to $0$ and} such that $\de_{0^{\8}}$ is an accumulation point for $\mu_{\be A+B_{\be}}$. 
\end{corollary}
\begin{proof}
 Set $D(x)=-d(x,0^{\infty})$. Then for any integer $n>0$, $\mu_{\be A+\frac1n D}$ goes to $\de_{0^{\8}}$ as $\be$ goes to $+\8$. 
 Therefore, one can find an increasing sequence $(\be_{n})\uparrow+\8$  such that for any $n$, 
 $$\mu_{\be_{n}A+\frac1n D}([0])\ge 1-\dfrac1n.$$
 Let $s:\mathbb{R}^+\to \mathbb{R}^+$ be a decreasing function satisfying $s(\beta_n)=\frac{1}{n}$. Then, $\mu _{\be_{n}A+s(\be_{n})D}$ goes to $\de_{0^{\8}}$ (as it is the unique invariant measure with support in $[0]$). 
\end{proof}

\subsection{{Kinds of Laplace's methods}  }
%%%%%

We remind the very easy fact: 
\begin{lemma}
\label{lem-summaxplus}
Let $a_{1},\ldots , a_{n}$ be real numbers. Then 
$$\lim_{\be\to+\8}\dfrac1\be\log\left(\sum_{i=1}^{n}e^{\be a_{i}}\right)=\max_{i}a_{i}.$$
\end{lemma}

%{\JM This can be extended to a more general setting and looks like  Laplace's method. }
\begin{lemma}[see  \cite{LMMS}]\label{lmms}
Let $M$ be any compact metric space. Let $W_{\be}:M\to\R$ be a family of continous functions  converging uniformly to $W:M\to \mathbb{R}$, as $\beta\to +\infty$. 
	Then, for any finite measure $\nu$ with $\supp(\nu) = M$ we have \[ \lim_{\beta \to \infty} \frac{1}{\beta} \log \int_M e^{\beta W_\beta(a)} d\nu(a) = \sup_{a\in M}\, W(a).  \]
\end{lemma}

\begin{proof}
	Let $m = \sup \{W(a) \,|\, a \in M\}$. We first prove 
	\[  \liminf_{\beta \to +\infty} \frac{1}{\beta} \log \int_M e^{\beta W_\beta(a)} d\nu(a) \geq m.  \]

	Given $\varepsilon>0$ let $\bar a \in M$ be such that $W(\bar a) > m-\varepsilon/2$.
	There exist $\beta_0$ and $\delta$, such that $W_\beta(a) > m - \varepsilon$, for any $a \in B(\bar a,\delta) = \{a \,|\, d(a,\bar a)<\delta \}$ and $\beta > \beta_0$.

	Therefore, if $\beta > \beta_0$, we have  that
	
	$$\int_M e^{\beta W_\beta(a)} d\nu(a) \geq \int_{B(\bar a,\delta)} e^{\beta W_\beta(a)} d\nu(a) > \nu\big(B(\bar a,\delta)\big) e^{\beta(m-\varepsilon)},$$
	and then
	\[  \frac{1}{\beta} \log \int_X e^{\beta W_\beta(a)} d\nu(a) >\frac{1}{\beta}\log\big(\nu\big(B(\bar a,\delta)\big)\big) + m-\varepsilon.\]
	We conclude that
	\[  \liminf_{\beta \to +\infty} \frac{1}{\beta} \log \int_M e^{\beta W_\beta(a)} d\nu(a) \geq m.  \]
	
	{The converse inequality is obtained in the same way: for a given} $\varepsilon$, there exists $\beta_0$ such that, $W_\beta(a) < m + \varepsilon$, for any $\beta>\beta_0$ and $a\in M$.
\end{proof} 

%%%%%%%%
%%%%%%%
\subsection{Eigenfunctions and calibrated subactions}

%%%%%
\subsubsection{Calibrated subaction as limit of eigenfunctions}\label{sssec-calisuba1}

\begin{lemma}\label{existsubact}  Let $B_\beta:X\to\mathbb{R}$ be a family of  Lipschitz functions satisfying
\begin{enumerate}
\item $Lip(B_\beta) \leq \beta c$ for some constant $c$ and any $\beta$;
\item $\frac{|B_\beta|_\infty}{\beta}$ converges to zero as $\beta \to\infty$.
\end{enumerate}
{ There exist accumulation points for $\frac{1}{\beta}\log(H_{\beta A+B_{\beta}})$ {for} the norm  $|\quad|_\infty$. Any such limit} is a calibrated subaction for $A$.
%
%  and such that . Let $A:X\to\mathbb{R}$ be a Lipschitz function. For any increasing sequence $\{\beta_i\}$ such that $\beta_i\to\infty${\JM, there exists at least} one accumulation point of $\frac{1}{\beta_i}\log(H_{\beta_i A+B_{\beta_i}})$ using the norm $|\cdot|_\infty$. Any accumulation point is Lipschitz and a calibrated subaction to $A$.
\end{lemma}
\begin{proof}
It is known that there exists a constant $c_1$ such that for any Lipschitz function $\phi$ and any $x, y\in X$
$$e^{-c_1(1+|\phi|_\infty +Lip(\phi))d(x,y)}\le \frac{H_\phi(x)}{H_\phi(y)} \leq e^{c_1(1+|\phi|_\infty +Lip(\phi))d(x,y)}$$
(see for example \cite{M} pp. 46-51, also see \cite[Lemme 1.8]{Bowen}). 
The main point here is that $c_{1}$ is an \emph{universal} constant and does not depend on $\phi$. Actually it only depends on the mixing rate in $X$. 

This yields that $\log H_{\phi}$ is Lipschitz continuous with Lipschitz constant bounded by $c_1(1+|\phi|_\infty +Lip(\phi))$. 
Applying this to $\phi:=\be A+B_{\be}$, we get that the family $\dfrac1\be\log H_{\be A+B_{\be}}$ is equicontinuous with bounded Lipschitz {norm}. 

By Arzel\`{a}-Ascoli theorem  we conclude that there exists accumulation points for the $|\quad |_{\8}$-norm and they are Lipschitz continuous. 

Note the next double-inequality
\begin{equation}
\label{eq-pressbb}
P(\be A)-|B_{\be}|_{\8}\le P(\be A+B_{\be})\le P(\be A)+|B_{\be}|_{\8}. 
\end{equation}
Now, by definition of the eigenfunction we get that for any $x\in X$:
$$\sum_{\sigma(z)=x} e^{\beta A(z) +\log(H_{\beta A +B_\beta}(z)) - \log(H_{\beta A +B_\beta}(x)) - P(\beta A +B_\beta)} =1,$$
holds. Taking $\dfrac1\be\log$, considering any accumulation point $\wt V$ for $\dfrac1\be\log H_{\be A+B_{\be}}$, Lemma \ref{lmms} and \eqref{eq-pressbb} yield
$$\sup_{\sigma(z) = x} \left[A(z) + \tilde{V}(z) - \tilde{V}(x) -m(A) \right]=0.$$

%
%The normalization of $H_A$ is not relevant in above inequality, but, as we consider $H_A(y)=1$ at some point $y$, then
%\[e^{-c_1(1+|A|_\infty +Lip(A))} \leq H_A \leq e^{c_1(1+|A|_\infty +Lip(A))}.\]
%Applying this inequalities on the functions $\beta_i A +B_{\beta_i}$, taking $\frac{1}{\beta_i}\log(\cdot)$ in both sides and using the Arzel\`{a}-Ascoli theorem we conclude that exists an accumulation point of $\frac{1}{\beta_i}\log(H_{\beta_i A+B_{\beta_i}})$ which is also Lipschitz.  Suppose now that $\lim_{\beta_j\to\infty} \frac{1}{\beta_j}\log(H_{\beta_j A+B_{\beta_j}})=\tilde{V}$. In order to prove that $\tilde{V}$ is a calibrated subaction we consider the equation 
%\[ \sum_{\sigma(z)=x} e^{\beta A(z) +\log(H_{\beta A +B_\beta}(z)) - \log(H_{\beta A +B_\beta}(x)) - P(\beta A +B_\beta)} =1.\]
%Taking  $\lim_{\beta_j \to\infty} \frac{1}{\beta_j}\log(\cdot)$  in both sides of this equation we get\footnote{observe that $ \frac{1}{\beta}P(\beta A) - \frac{|B_\beta|_\infty}{\beta} \leq \frac{1}{\beta}P(\beta A +B_\beta) \leq \frac{1}{\beta}P(\beta A) + \frac{|B_\beta|_\infty}{\beta}$}
%\[\sup_{\sigma(z) = x} A(z) + \tilde{V}(z) - \tilde{V}(x) -m(A) =0.\]
%This shows that $\tilde{V}$ is a calibrated subaction to $A$.

\end{proof}

%%%%%%
\subsubsection{Calibrated subactions and the Ma\~ né Potential}
We remind that the Aubry set $\Om$ is non-empty as it contains the Mather set. 
First, we state a technical lemma.
\begin{lemma}
\label{lem-techmane}
Let $x$ and $y$ be in $X$. Let $\ol\xi_{n}$ be points such that $\ol\xi_{n}\to x$ as $n$ goes to $+\8$, and for any $n$ there exists $k_{n}$ with $k_{n}\to+\8$ as $n\to+\8$ such that $\s^{k_{n}}(\ol\xi_{n})=y$. Then, 
$$S(x,y)\ge\limsup_{n\to+\8}S_{k_{n}}(A)(\ol\xi_{n}).$$
\end{lemma}
\begin{proof}
Pick $\eps>0$. Let $N$ be such that for every $n>N$, $d(x,\ol\xi_{n})<\eps$. Then 
$$S^{\eps}(x,y):=\sup{\{S_{k}(A)(z),\ \s^{k}(z)=y,\ d(x,z)<\eps\}}\ge S_{k_{n}}(A)(\ol\xi_{n})$$
holds as soon as $n>N$ holds.
This yields, for every $\eps$, $\disp S^{\eps}(x,y)\ge\limsup_{n}S_{k_{n}}(A)(\ol\xi_{n})$. Hence doing $\eps\downarrow 0$ we get the result. 
\end{proof}

\begin{lemma}
\label{lem-AonOm}
The potential $A$ is equal to 0 on $\Om$. 
\end{lemma}
\begin{proof}
By definition $x$ belongs to $\Om$ if and only if $S(x,x)=0$. Hence we get 
$$ 0\ge A(x)\geq S(x,x)=0.$$
\end{proof}

Now we present a very important characterization of calibrated subactions using Ma\~{n}\'e Potential and Aubry and Mather sets.

\begin{lemma}\label{mane}
For any calibrated subaction $V$ we have
\[V(x_0) = \sup_{a\in \mathcal{M}} [V(a) +  S(a,x_0)] =\sup_{a\in \Omega} [V(a) +  S(a,x_0)]\,\,\forall \,\, x_0 \in X.\]
\end{lemma}
\begin{proof}
We follow  ideas of \cite{BLT}. Let $V$ be any calibrated subaction. As $\mathcal{M}\subset \Omega$ holds, we get
\[\sup_{a\in \mathcal{M}} [V(a) +  S(a,x_0)] \leq \sup_{a\in \Omega} [V(a) +  S(a,x_0)].\]
As $V$ is a subaction, $A +V-V\circ\sigma \leq 0$ and then, for any $n$ and any $z$:
$$S_{n}(A)(z)+V(z)\le V\circ \s^{n}(z).$$
This yields for any  $a,x_0 \in X$,  $S(a,x_0)+V(a)-V(x_0)\leq 0$, hence
\[V(x_0) \geq V(a) +  S(a,x_0)\,\,\forall a,x_0 \in X,\]
holds. This proves that 
\[V(x_0) \geq \sup_{a\in \Omega} [V(a) +  S(a,x_0)] .\]

Now we will prove that $\sup_{a\in \mathcal{M}} [V(a) +  S(a,x_0)] \geq V(x_0)$.  Set $R_{-}=A +V -V\circ \sigma$.  
By definition of calibrated subaction the following holds: 
\begin{enumerate}
\item $R_{-} \leq 0$,
\item There exists a sequence of points $(x_n)$ such that { for every $n\geq 1$}, $\sigma(x_n) = x_{n-1}$ and   $R_{-}(x_n) = 0$. 
\end{enumerate}
Let us consider the sequence of probabilities $\nu_n = \frac1n\sum_{i=1}^{n}\de_{x_{i}}$, $n\ge 1$. { Let $\nu$ be an accumulation point  for the weak* topology as $n\to\infty$,   say $\nu:=\lim_{n_j\to\infty} \nu_{n_j}$}. As for any continuous function $f:X\to\mathbb{R}$ we have
\begin{align*}
\int f\circ\sigma\,d\nu &= \lim_{n_j} \frac{1}{n_j}\sum_{i=1}^{n_j}f(\sigma(x_i)) \stackrel{\sigma(x_i) = x_{i-1}}{=} \lim_{n_j} \frac{1}{n_j}\sum_{i=1}^{n_j}f(x_{i-1})\\
&=\lim_{n_j} \frac{1}{n_j}[f(x_0)-f(x_{n_j})+\sum_{i=1}^{n_j}f(x_{i})] = \lim_{n_j} \frac{1}{n_j}[\sum_{i=1}^{n_j}f(x_{i})]=\int f\,d\nu,
\end{align*} 
then $\nu$ is $\sigma-$invariant. 

 By construction { of $R_{-}$ and $\nu_n$ we have} $\disp\int A\,d\nu_{n}=\int R_{-}\,d\nu_n=0$, thus $\disp\int A\,d\nu=0$. Therefore $\nu$ is $A$-maximizing.

Now we fix any point $a\in \operatorname{supp}(\nu)$. Note that $a$ belongs to $\CM\subset\Om$.  By definition of support we have $\nu( \{z|d(z,a)\leq \epsilon\})>0$ for any $\epsilon>0$. It follows that, for arbitrarily large $n$ we have also $\nu_n(\{z:d(z,a)\leq \epsilon\}) >0$, which proves that $a$ is an accumulation point of the sequence $(x_n)$. By considering a subsequence $(x_{n_k})$ satisfying  $x_{n_k}\to a$ we get
\begin{align*}
 S(a,x_0) &\geq \lim_{n_k\to\infty}[S_{n_k}(A)(x_{n_k}) ]\\
&=\lim_{n_k\to\infty} [S_{n_k}(R_{-})(x_{n_k}) -V(x_{n_k})+V(x_0)]\\
& =V(x_0)-V(a).
\end{align*}
This proves 
\[V(x_0) \leq \sup_{a\in \mathcal{M}} [V(a) +  S(a,x_0)].\]
\end{proof}

%%%%%%%%%%%
%%%%%%%%%%%%%%
%%%%%%%%%%%%
\section{Proof of Theorem \ref{theoA}}\label{secproofthA}
We remind that $A$ is assumed to be Lipschitz continuous, non-positive and satisfying $m(A)=0$. 
The Aubry set is defined by 
$$\Om:=\{x\in X,\ S(x,x)=0\},$$
where $S(.,.)$ is the Ma\~né potential (see Def. \ref{def-maneaubrymather}). 

%%%%%
%%%%%%
\subsection{Technical results for the irreducible pieces of $\Omega$}

%\subsection{Proof when all the irreducible components have maximal entropy}

Let us consider a decomposition of the Aubry set in irreductible sub-shifts of finite type 
\[\Omega = \Sigma_1 \cup ...\cup \Sigma_k \cup\Sigma_{k+1}\cup...\cup \Sigma_{L}\]
where $\Sigma_i \cap \Sigma_j =\emptyset$ for $i\neq j$, $h_{top}(\Sigma_i) = h$ if $i\in\{1,...,k\}$ and $h_{top}(\Sigma_i)<h$ if $i\in\{k+1,...,L\}$.

It follows that $A=0$ over the sets $\Sigma_i$ where $i\in\{1,...,L\}$.

\begin{lemma}\label{lem-compact}
	For each $i\in\{1,...,L\},$ the set $\s^{-1}(\S_{i})\setminus\S_{i}$ is compact.
\end{lemma}

	\begin{proof}
		As $X$ is compact, we only need to prove that 	$\s^{-1}(\S_{i})\setminus\S_{i}$ is closed. Let $(y_n)$ be a sequence in $\s^{-1}(\S_{i})\setminus\S_{i}$  converging to  $y$ in $X$. We will prove that $y\in  \s^{-1}(\S_{i})\setminus\S_{i}$.	
		
		As $\S_{i}$ is a sub-shift of finite type we can find some positive integer $l$ and a collection of cylinders of length $l$, say $\CA$, coinciding with the admissible words (of length $l$) for $\S_{i}$. That is that $x$ belongs to $\S_{i}$  if and only if for every $n\ge 0$, $\s^{n}(x)$ belongs to one of the cylinders in $\CA$. 
		
		By continuity, $\s(y_{n})$ converges to $\s(y)$. Furthermore, $\s(y_{n})$ belongs to the closed set $\S_{i}$, by hypothesis, and thus $\s(y)$ also belongs to $\S_{i}$. 
		
		As $y_n\to y$, we can assume that all the $y_{n}$ belong to the same $l$-cylinder than $y$. As $\s(y_n) \in \S_{i}$ and $y_n\notin \S_i$ this cylinder does not belong to the collection $\CA$  and so $y\notin \S_i$.  Therefore, $y\in 	\s^{-1}(\S_{i})\setminus\S_{i}$
	\end{proof}

%\begin{proof}
%	{\JM to check again or to use the above new version} 
%	
%	Let $x$ be in $\S_{i}$. Because $\S_{i}$ is a subshift of finite type, we can consider a cylinder $[\om_{0}\ldots \om_{N}]\ni x$ such that for any $\S_{i}$-eligible  word $\om_{1}\ldots \om_{N}\om_{N+1}$, $\om_{1}\ldots \om_{N}\om_{N+1}$ is also eligible. 
%	
%	Hence, if $y$is such that $\s(y)\in \S_{i}$ but $y\notin\S_{i}$, then $y$ can not belong to $[\om_{0}\ldots \om_{N}]$, since $\s(y)=\om_{1}\ldots \om_{N}y_{N+1}\ldots$ is $\S_{i}$-eligible. 
%	
%	
%	%Assume we get a sequence of $y_{n}\notin\S_{i}$, $\s(y_{n})\in \S_{i}$, such that $y_{n}\to x\in \S_{i}$. We can pick cylinder $[\om_{0}\ldots \om_{N}]$, such that the associated word $\om_{0}\ldots \om_{N}$ is not eligible in $\S_{i}$ and all the $y_{n}$ (for sufficiently large $n$) belong to this cylinder. This yields that $x$ also belong to that cylinder which is absurd. 
%	
%	Therefore, for any $x\in \S_{i}$ one can find a ball $B(x,\eps_{x})$ containing $x$ and having empty intersection with $\s^{-1}(\S_{i})\setminus\S_{i}$. If $x$ runs over $\S_{i}$, this form a cover from which one can extract a finite cover. This proves that $\s^{-1}(\S_{i})\setminus\S_{i}$ is contained in a closed set disjoint from $\S_{i}$, say $C$. Therefore $\s^{-1}(\S_{i})\setminus\S_{i}=\s^{-1}(\S_{i})\cap C$ is also a compact set.  
%\end{proof}
%

%We assume that {\JM $A\leq 0$ and as $m(A)=0$ we obtain also $A(x)=0$ for any $x\in \Omega$. Up to a coboundary (and a constant) we can always do this assumption. }  

%%%%%%
%%%%%%
\subsection{Another characterization for the $\S_{i}$'s}
We recall that the Aubry set $\Omega$ is defined as the set of points $x$ satisfying 
$S(x,x)=0,$  where the function $S$ is the Ma\~{n}\'{e} potential. 

We define an equivalence relation by 
$$x \sim y\iff S(x,y)+S(y,x)=0.$$

Our main result for this section (Prop. \ref{prop-componentenmatch}) is that equivalences classes for that relation coincide with the decomposition of $\Omega$ in irreducible components.

Let us first recall a classical result for the Ma\~né potential:
	
	\begin{lemma} \label{lema : Mane} Let $S$ be the Ma\~{n}\'{e} potential. We have:
	\begin{enumerate}
\item  $S(x,y)+S(y,z)\le S(x,z)$ for every $x$, $y$, $z$ in $X$;
\item for any fixed $x\in X$, the map $y\mapsto S(x,y)$ is Lipschitz continuous;
\item if $x\sim y$, then $S(x,.)=S(y,.)$;
\item if $\s(x)=y$, then $S(x,y)=A(x)$.
\end{enumerate}
\end{lemma}

\begin{proof} For items 1 and 2 see \cite{BLL}. 

In order to prove 3, just observe that $S(x,y)+S(y,x)=0$ and $S(.,.)\le 0$ yields 
$$S(x,y)=S(y,x)=0.$$ 
Hence, for any $z\in X$,
	\[S(x,z) \stackrel{1.}{\geq} S(x,y)+S(y,z) = S(y,z)  \stackrel{1.}{\geq} S(y,x)+S(x,z) = S(x,z).\]
	
The proof of item 4 has two steps. First, note that if $	z$ is such that $\s^{n}(z)=y$ and $d(x,z)<\eps$, then 
$$S_{n}A(z)\le A(z)\le A(x)+Cd(x,z) \le A(x)+C\eps.$$
{This yields for any positive $\eps$:
$$S^{\eps}(x,y)=\sup\{S_n(A)(z),\ \s^{n}(z)=y,\ d(x,z)<\eps\}\le { A(x)+C\eps}.$$
Doing $\eps\downarrow0$ we get $S(x,y)\le A(x)$. 

\medskip
On the other hand, as $\s(x)=y$, 
$$\sup\{S_n(A)(z),\ \s^{n}(z)=y,\ d(x,z)<\eps\}\ge A(x),$$
which yields the other inequality. }
\end{proof}

\begin{proposition}
	\label{prop-componentenmatch}
	The irreducible components {$\Sigma_i$, $i\in\{1,...,L\}$,} of the Aubry set as sub-shift of finite type coincide with the classes of  equivalence, which are defined by the equation 
	$$S(x,y)+S(y,x)=0.$$
\end{proposition}

	\begin{proof} 	Let $x=(x_1,x_2,x_3,...)$ and $y=(y_1,y_2,y_3,...)$ be points of $X$.
		
		\textbf{Step one:} Suppose $x,y\in \Sigma_i$. { We  show that $S(x,y) = 0$.

	\medskip	
		As $\S_i$ is irreducible, for each $n$ there exists a word $z_1^n,...,z_{N(n)}^n$ such that 
		\[\ol\xi_{n} := (x_1,...,x_n,	z_1^n,...,z_{N(n)}^n, y_1,y_2,y_3,...) \in \S_i.\]
		As $\S_i\subset \Omega$ we get $S_{n+N(n)}(A)( \ol\xi_{n}) = 0$, $\lim_{n\to+\8}\ol\xi_{n}=x$, and $\s^{n+N_{n}}(\ol\xi_{n})=y$. Then Lemma \ref{lem-techmane} yields 
		$$0\ge S(x,y) \geq \lim_{n\to +\infty} S_{n+N(n)}(A)( \ol\xi_{n}) = 0.$$
		This shows that $S(x,y)=0$ holds. Similarly, exchanging the roles for $x$ and $y$ we get $S(y,x)=0$. 
		
\bigskip		
		\textbf{Step two:} Suppose $S(x,y)+S(y,x) = 0$. Item 1 in Lemma \ref{lema : Mane} yields $S(x,x)=S(y,y)=0$, hence $x,y\in \Omega$. 
		
		Let us suppose by contradiction that $y\in \S_i$ and $x\notin \S_i$.  As $S(x,y) = 0$ there exists a sequence of points $\ol\xi_{n}:=(x_{N(n)}^n,...,x_{1}^n,y_1,y_2,y_3,...), n\in \mathbb{N},$ which converges to $x$, such that $S_{N(n)}(A)(\ol\xi_{n})>-\frac{1}{n}.$ 
		
		As $\S_i$ is closed and $x\notin \S_i$ we have that $\ol\xi_{n} \notin \S_i$ for $n$ large enough. For simplicity we suppose that for every $n$, $\ol\xi_{n} \notin \S_i$. 
		
		As $y \in \S_i$ and $\S_{i}$ is invariant, we can consider $M(n)\le N(n)$ such that 
		$$\ol z_{n}:=\s^{N(n)-M(n)}(\ol\xi_{n})\notin \S_{i},\ \s(\ol z_{n}):=\s^{N(n)-M(n)+1}(\ol\xi_{n})\in \S_{i}.$$
		Then $\ol z_n \in \s^{-1}(\S_{i})\setminus\S_{i}$ which is compact, therefore there exists a converging subsequence for $(\ol z_n)$. Again, for simplicity we assume that $\ol z_{n}\to z$ holds. Note that $z \in  \s^{-1}(\S_{i})\setminus\S_{i}$ holds. Furthermore we have
		$$0\ge S_{M(n)}(A)(\ol z_{n})\ge S_{N(n)}(A)(\ol \xi_{n})>-\frac{1}{n}.$$
Then, Lemma \ref{lem-techmane} yields $0\ge S(z,y)\ge 0$, hence $S(z,y)=0$.

\medskip

As $\ol\xi_{n}$ converges to $x$ and $\s^{N(n)-M(n)}(\ol\xi_{n})=\ol z_{n}$ we get for any positive $\eps$ and for any sufficiently big $n$,  
$$S^{\eps}(x,\ol z_{n})\ge S_{N(n)-M(n)}(A)(\ol\xi_{n})>-\frac1n.$$
Then,   Lemma \ref{lem-techmane}  and Lipschitz continuity in the second variable for $S(x,.)$ yield
$$0\ge S(x,z)=\lim_{n}S(x,\ol z_{n})=0.$$

Finally, $S(z,z)\ge S(z,y)+S(y,x)+S(x,z)=0$ shows that $z$ belongs to some $\S_{k}$. So does $\s(z)$, hence $k=i$ and this produces a contradiction.

%
%
%
%		
%		 which converges to a point $z \in  \s^{-1}(\S_{i})\setminus\S_{i}$ (we can suppose $z_n\to z$ as $n\to\infty$)
%		
%		
%		
%		and $y_n \notin \S_i$ there exists $z_n:=(x_{M(n)}^n,...,x_{1}^n,y_1,y_2,y_3,...) $, where $M(n)\leq N(n)$, satisfying
%		$z_n \notin \S_i$ and $\sigma(z_n)\in \S_i$. Then $z_n \in \s^{-1}(\S_{i})\setminus\S_{i}$ which is compact, therefore there exists a subsequence of $(z_n)$ which converges to a point $z \in  \s^{-1}(\S_{i})\setminus\S_{i}$ (we can suppose $z_n\to z$ as $n\to\infty$). We have
%		$S(z,y) \geq \lim_{n\to+\infty}S_{M(n)}(A)(z_n) \geq \lim_{n\to+\infty}-\frac{1}{n}=0$ and then $S(z,y)=0$. We claim that $S(x,z) = 0$ also holds.
%		
%		Indeed, if $M(n)=N(n)$ for infinitely many $n$ then $z=x$ and the claim is consequence of the hypothesis of step two.
%		If $N(n)>M(n)$ for $n$ large enough then, denoting {\JM $z=(w_1,w_2,w_3,...)$}, we have
%		\[S(x,z) \geq \lim_{n\to +\infty} S_{N(n)-M(n)}(A)(x_{N(n)}^n,...,x_{M(n)+1}^n,w_1,w_2,w_3,...)\]\[\geq\lim_{n\to +\infty} [ S_{N(n)-M(n)}(A)(y_n) -Cd(z,z_n)] \geq \lim_{n\to +\infty} [-\frac{1}{n}-Cd(z,z_n)]=0.\]
%		
%		
%		As $S(y,x)=S(z,y)=S(x,z)=0$ we obtain
%		\[S(z,z) \geq S(z,y) + S(y,x) + S(x,z) = 0\]
%		and, therefore, $z\in \Omega$. As $\Omega = \S_1 \cup \S_2\cup..\cup \S_L$ and $z\notin \S_i$ we have $z\in \S_j$ with $j\neq i$. As $\S_j$ is $\sigma-$invariant we conclude that $\sigma(z) \in \S_j$. This is a contradiction because $z\in \s^{-1}(\S_{i})\setminus\S_{i}$.
		
}\end{proof}

\begin{lemma}\label{V(sigma)}
If $V$ is a calibrated subaction for $A$ then $V$ is constant on each set $\Sigma_i$.	
\end{lemma}

\begin{proof} 
	Let $x,y \in \S_i$ and $V$ be a calibrated subaction.
	As $S(x,y)=S(y,x) = 0$,  by applying Lemma \ref{mane} we have
	\[ V(x) =\sup_{a\in \Omega}[S(a,x)+V(a)] \geq S(y,x) +V(y) = V(y).\]
and	\[V(y) = \sup_{a\in \Omega}[S(a,y)+V(a)] \geq S(x,y) +V(x) = V(x).\]
\end{proof}

%%%%%%
%%%%%%%
\subsection{A Max-Plus formula for the calibrated subaction $V$}

{ From now on, let us denote $S_{i}(.) := S(x,.),$ where $x$ is any point of $\S_i$ (see Lemma \ref{lema : Mane} item 3.). 
	If $\S_{i}$ and $\S_{j}$ are two irreducible components of $\Omega$ we set 
	$$a_{ij}:=\sup_{x\in \S_{i}}\sup_{y\in \s^{-1}(x),y\notin\S_{i}}A(y)+S_{j}(y)=\sup_{y\in \s^{-1}(\S_{i})\setminus\S_{i}}A(y)+S_{j}(y).$$}

%------------
%
%{\RL Notation : we should write $h_{0}$ for topological entropy ??}{\JM Below it is used $h_{top}(\S_i)$.}
%
%-------------

We remember that we are supposing in this section that
\[\Omega = \Sigma_1 \cup ...\cup \Sigma_k \cup\Sigma_{k+1}\cup...\cup \Sigma_{L}\]
where $\Sigma_i \cap \Sigma_j =\emptyset$ for $i\neq j$, $h_{top}(\Sigma_i) = h$ if $i\in\{1,...,k\}$ and $h_{top}(\Sigma_i)<h$ if $i\in\{k+1,...,L\}$.

\begin{lemma}
	\label{lem-aijk}
	For every $i,j,l \in\{1,...,L\}$ the next holds:
	\begin{enumerate}
\item $-\8<a_{ij}\le 0$;
\item $a_{ii}<0$;
\item$a_{li}+a_{ij}\le a_{lj}$.
\end{enumerate}
\end{lemma}
\begin{proof} First, $A$ is non-positive, and thus so is $S$. Second, Lemma \ref{lema : Mane} yields each $S_{j}$ is Lipschitz continuous, hence bounded on the compact set  $\s^{-1}(\S_{i})\setminus\S_{i}$ {(see Lemma \ref{lem-compact})}. This also holds for $A$, and thus  each $a_{ij}$ is a real number. 
This finishes to prove 1. 	
	
	\bigskip
	
	\noindent
	Proof of 2. \newline 
	As $y\mapsto A({y})+S_{i}(y)$ is continuous on the compact set $\s^{-1}(\S_{i})\setminus\S_{i}$, it attains its maximum. 
	
	If $y\in \s^{-1}(\S_{i})\setminus\S_{i}$ is such that 
	$$a_{ii}=A(y)+S_{i}(y),$$
	then, setting  $x:=\sigma(y)$, we have $x\in \S_i$ and $a_{ii}=A(y)+S(x,y)$. 
	{By item 4 in Lemma \ref{lema : Mane},  $S(y,x)=A(y)$}. This yields, 
	$$a_{ii}={A(y)+S(x,y)= S(y,x)+S(x,y)}<0,$$
	since $y\notin\S_{i}$ (see Prop \ref{prop-componentenmatch}).  
	
	\bigskip
	
	\noindent
	Proof of 3. \newline 
	If $i=l$ we just need to apply item 2. Suppose then $i\neq l$ and let $y_{i}\in \s^{-1}(\S_{i})\setminus\S_{i}$ and $y_{l}\in \s^{-1}(\S_{l})\setminus\S_{l}$ be such that 
	$$a_{ij}=A(y_{i})+S_{j}(y_{i}),\quad a_{li}=A(y_{l})+S_{i}(y_{l}).$$ 
	Set $x_{i}:=\s(y_{i})\in \S_{i}$ and 
	note that $S_{i}(y_{l})=S(x_{i},y_{l})$ holds.  Again Lemma  \ref{lema : Mane} yields $S(y_{i},x_{i})=A(y_{i})$. Hence we get 
	\begin{eqnarray*}
		a_{li}+a_{ij}&=& A(y_{l})+S_{i}(y_{l})+A(y_{i})+S_{j}(y_{i})\\
		&= & A(y_{l})+S(x_{i},y_{l})+S(y_{i},x_{i})+S_{j}(y_{i})\\
		&\le & A(y_{l})+S_{j}(y_{l})\le a_{lj}. 
	\end{eqnarray*}
\end{proof}

\begin{remark}
	If $V$ is a calibrated subaction, then it is constant on each set $\S_{i}$ (by Lemma \ref{V(sigma)}). It thus makes sense to set $V(\S_{i})$ and we will use such notation from now on in the present section.
\end{remark}

{

\begin{proposition} 
Any accumulation point for $ \frac{1}{\beta}\log(P(\beta A) - h )$ belongs to $]-\8,0]$. Furthermore, if $\displaystyle{\lim_{\beta_j\to\infty} \frac{1}{\beta_j}\log(P(\beta_j A) - h )}=\ga$ and $\disp\lim_{\beta_j\to\infty} \frac{1}{\beta_{j}}\log(H_{\beta_{j}A})=V$, 
then, 
 \begin{equation}\label{eq - gamma and V} \gamma + V(\Sigma_i) = \max_{j\in\{1,...,L\}}V(\Sigma_j)+a_{ij},\,\,\,\forall\,i\in\{1,...,k\}\end{equation} 
 and
	\begin{equation}\label{eq - gamma and V - outra} V(\Sigma_i) = \max_{j\in\{1,...,L\}}V(\Sigma_j)+a_{ij},\,\,\,\forall\,i\in\{k+1,...,L\}.\end{equation} 
\end{proposition}

\begin{proof} 

As it is said above (see Lemma \ref{pressureselect}), $\be\mapsto P(\be A)-h$ is non-increasing  and goes to 0 as $\be\to+\8$. Hence, $\be\mapsto\disp \frac1\be\log(P(\be A)-h)$ is also non-increasing  and it is negative for sufficiently large $\be$. This shows that any accumulation point for $\disp\frac1\be\log(P(\be A)-h)$ is in $[-\8,0]$.  

Because $V$ is Lipschitz continuous on $\Om$, thus bounded, and beause all the $a_{ij}$ are real numbers, Equation \eqref{eq - gamma and V} shows that $\ga$ is a real number, hence $\ga>-\8$. 

\bigskip
Let us now consider the Ruelle operator for $\beta A$, given by $${L}_{\beta A}(\varphi)(x) = \sum_{y\in \sigma^{-1}(x)}e^{\beta A(y)}\varphi(y).$$

For each $i\in\{1,...,L\}$,  we can also consider the Ruelle operator for the zero potential acting over $\CC(\Sigma_i)$:
$$\mathcal{L}_{i}(\varphi)(x) = \sum_{y\in \sigma^{-1}(x),\,y\in \Sigma_i}\varphi(y)$$ where $x\in \Sigma_i$.
As we suppose that  $\Sigma_i$ is an irreducible sub-shift of finite type, there exists an eigen-measure $\nu_i$ with $\supp(\nu_i) = \Sigma_i$ such that
\[e^{h_i} \int \varphi(x) d\nu_i(x) = \int \mathcal{L}_{i}(\varphi)(x) \,d\nu_i(x),\]
where $h_i$ is the topological entropy of $\Sigma_i$, which corresponds to the pressure of the zero function on $\Sigma_i$.

We start by supposing $i\in\{1,...,k\}$.	Given $x\in \Sigma_i$ we have
	\[e^{P(\beta A)}H_{\beta A}(x) = \sum_{y\in \sigma^{-1}(x),\,y\in \Sigma_i}e^{\beta A(y)}H_{\be A}(y) +  \sum_{y\in \sigma^{-1}(x),\,y\notin \Sigma_i}e^{\beta A(y)}H_{\be A}(y)\]
	and using that $A=0 $ in $\Sigma_i$ we get
	\[e^{P(\beta A)}H_{\beta A}(x)=\mathcal{L}_{i}(H_{\be A})(x) +  \sum_{y\in \sigma^{-1}(x),\,y\notin \Sigma_i}e^{\beta A(y)}H_{\be A}(y).\]
	Integrating both sides with respect to the eigenmeasure $\nu_i$ of $\mathcal{L}_{i}$ we have
	\[e^{P(\beta A)}\int H_{\beta A}(x)\,d\nu_i(x) = \int \mathcal{L}_{i}(H_{\be A})(x)\,d\nu_i(x) +  \int\sum_{y\in \sigma^{-1}(x),\,y\notin \Sigma_i}e^{\beta A(y)}H_{\be A}(y)\,d\nu_i(x)\]
	{ i.e.}
	\[e^{P(\beta A)}\int H_{\beta A}(x)\,d\nu_i(x) = e^{h}\int H_{\be A}(x)\,d\nu_i(x) +  \int\sum_{y\in \sigma^{-1}(x),\,y\notin \Sigma_i}e^{\beta A(y)}H_{\be A}(y)\,d\nu_i(x).\]
	Then
	\begin{equation}\label{ja1}
		(e^{P(\beta A)-h}-1)e^{h}\int H_{\beta A}(x)\,d\nu_i(x) =  \int\sum_{y\in \sigma^{-1}(x),\,y\notin \Sigma_i}e^{\beta A(y)}H_{\be A}(y)\,d\nu_i(x).
	\end{equation}
	Taking $\lim_{\beta_{j}\to+\infty}1/\beta_{j}\log(\cdot)$ in both sides of \eqref{ja1} we get
	\begin{equation}\label{ja2}
		\gamma +0+V(\Sigma_i) = \sup_{x\in \Sigma_i}\sup_{y\in\sigma^{-1}(x),y\notin\Sigma_i} A(y) + V(y),
	\end{equation}
	where for both integrals we use Lemma \ref{lmms} and that $\nu_i$ has  full support in $\Sigma_i$.

%	We remark that, as $P(\beta A)$ decreases to $h$, the map  $\beta \to \frac{1}{\beta}\log(P(\beta A) - h ) $ is decreasing. Furthermore, it is bounded below. Indeed, if the  map $\beta \to \frac{1}{\beta}\log(P(\beta A) - h )$ is not bounded below, then there exists a sequence $\beta_j$ such that $\lim_{\beta_j\to\infty} \frac{1}{\beta_j}\log(P(\beta_j A) - h ) = -\infty$ and with similar arguments we get a new equation \eqref{ja2} with $\gamma = -\infty$ which is impossible because $V$ and $A$ are limited.

\bigskip
Now, it follows from Lemmas \ref{mane} and \ref{V(sigma)} that  $$V(x) = \max_{j\in\{1,...,L\}} V(\Sigma_j) +S_j(x)$$ holds where $S_j$ is the Ma\~{n}\'{e} potential with respect to $\Sigma_j$.
	Therefore,  Equation \eqref{ja2} can be rewritten as
	\begin{align*}
		\gamma + V(\Sigma_i) &= \sup_{x\in \Sigma_i}\sup_{y\in\sigma^{-1}(x),y\notin\Sigma_i} A(y) + \left(\max_{j\in\{1,...,k\}} V(\Sigma_j) +S_j(y)\right)\\
		&= \sup_{x\in \Sigma_i}\sup_{y\in\sigma^{-1}(x),y\notin\Sigma_i} \max_{j\in\{1,...,L\}}( A(y) +  V(\Sigma_j) +S_j(y))\\
		&= \max_{j\in\{1,...,L\}} \sup_{x\in \Sigma_i}\sup_{y\in\sigma^{-1}(x),y\notin\Sigma_i} ( A(y) +  V(\Sigma_j) +S_j(y))\\
		&=  \max_{j\in\{1,...,L\}}V(\Sigma_j)+\left(\sup_{x\in \Sigma_i}\sup_{y\in\sigma^{-1}(x),y\notin\Sigma_i} A(y)+ S_j(y)\right).
	\end{align*} 
	Then we have
	\begin{equation*} \gamma + V(\Sigma_i) = \max_{j\in\{1,...,L\}}V(\Sigma_j)+a_{ij}.\end{equation*} 
	
\bigskip	
The proof for the case $i\in\{k+1,...,L\}$ is similar, except that $\ga$ disappears.  Indeed, if $\S_{i}$ has entropy $h'<h$, then \eqref{ja1} can be rewritten as:
	\begin{equation*}
		(e^{P(\beta A)-h'}-1)e^{h'}\int H_{\beta}(x)\,d\nu_i(x) =  \int\sum_{y\in \sigma^{-1}(x),\,y\notin \Sigma_i}e^{\beta A(y)}H_{\be A}(y)\,d\nu_i(x).
	\end{equation*}
	The rest of the computations holds, except that $\gamma$ has to be removed. 
	Hence, we get:
	\begin{equation*}
		V(\S_{i})=\max_{j\in\{1,...,L\}}V(\S_{j})+a_{ij}.
	\end{equation*}
	
	\end{proof} }

%%%%%%%%
%%%%%%%%
\subsection{Component with small entropy do not play a role}

We remember that $h_{top}(\Sigma_i) = h$ if $i\in\{1,...,k\}$ and $h_{top}(\Sigma_i)<h$ if $i\in\{k+1,...,L\}$.
\begin{proposition}
	\label{prop-maxplusSiSJ}
	If $i\in\{k+1,...,L\}$ then  for every $l\in\{1,...,L\}$, there exists $j\in\{1,...,k\}$ such that  
	%$$V(\S_{i})=\max_{j, \S_{j}\text{has entropy }h_{t}}\left(V(\S_{j})+a_{ij}\right).$$
	$$V(\S_{i})+a_{li}\le V(\S_{j})+a_{lj}.$$
\end{proposition}

	\begin{proof} 
			From equation \eqref{eq - gamma and V - outra} and item 2. of Lemma \ref{lem-aijk}, there exists $j_{1}\neq i$ such that 
		\begin{equation}\label{eq- 1}
			V(\S_{i})=V(\S_{j_{1}})+a_{ij_{1}}.\end{equation}
		Hence, item 3. of Lemma \ref{lem-aijk} yields 
		\begin{equation*}
			V(\S_{i})+a_{li} = V(\S_{j_{1}})+a_{ij_{1}}+a_{li} \le  V(\S_{j_{1}})+a_{lj_{1}}.
		\end{equation*}
		If $j_{1}\in\{1,...,k\}$ then the proposition is proved. 
		If not, then we apply the same reasoning with $j_{1}$ instead of $i$. This produces $j_{2}\neq j_{1}$ such that 
		\begin{equation}\label{eq- 2}
			V(\S_{j_{1}})=V(\S_{j_{2}})+a_{j_{1}j_{2}}.
		\end{equation}
		Equations \eqref{eq- 1} and \eqref{eq- 2} along with Lemma \ref{lem-aijk} yields 
		\begin{equation}\label{eq- 3}
			V(\S_{i})=V(\S_{j_{2}})+a_{ij_{1}}+a_{j_{1}j_{2}}\le V(\S_{j_{2}})+a_{ij_{2}}.
		\end{equation} 
		%As $V(\S_{i})=\max_{j}\left(V(\S_{j})+a_{ij}\right)$ holds, we conclude that an equality is satisfied in \eqref{eq- 3}. Then
		%	$$V(\S_{i})= V(\S_{j_{2}})+a_{ij_{2}} \text{and}\,\, a_{ij_{1}}+a_{j_{1}j_{2}}=a_{ij_{2}}.$$
		Again, if $j_{2}\in\{1,...,k\}$, then we are done. Otherwise we continue the reasoning.  Since there are finitely many components, either we eventually find some $j_{l}$ in $\{1,...,k\}$, or some $j_{r}$ will be equal to $j_l$ with $l<r$. This last case is impossible because we reach a contradiction. Indeed, by applying a similar equation to \eqref{eq- 2} recursively from $j_l$, we get  
		\[ V(\S_{j_l})=V(\S_{j_{l+1}})+a_{j_{l}j_{l+1}}=...=V(\S_{j_{r}})+a_{j_{l}j_{l+1}}+...+a_{j_{r-1}j_{r}} \]
		Using that $j_r=j_l$ and applying items 2. and 3. of Lemma \ref{lem-aijk}, we obtain
		\[  V(\S_{j_l}) =V(\S_{j_{l}})+a_{j_{l}j_{l+1}}+...+a_{j_{r-1}j_{l}}\leq V(\S_{j_{l}})+a_{j_{l}j_{l}} < V(\S_{j_{l}})\]
		which is a contradiction.

\end{proof} 
%
%
%\begin{proof}{\JM old version} 
%	
%	Pick any $k$. 
%	Note that $V(\S_{i})=\max_{l}\left(V(\S_{l})+a_{il}\right)$. 
%	{\JM Lemma \ref{lem-aijk}} shows that there exists $j_{1}\neq i$ such that 
%	$$V(\S_{i})=V(\S_{j_{1}})+a_{ij_{1}}.$$
%	Hence, Lemma \ref{lem-aijk} yields 
%	\begin{eqnarray*}
%		V(\S_{i})+a_{ki}&=& V(\S_{j_{1}})+a_{ij_{1}}+a_{ki}\\
%		&\le & V(\S_{j_{1}})+a_{kj_{1}}.
%	\end{eqnarray*}
%	If $\S_{j_{1}}$ has topological entropy $h$ then the proposition is proved. 
%	
%	If not, then we {\JM apply} the same reasoning with $j_{1}$ instead of $i$. This produces $j_{2}\neq j_{1}$ such that 
%	\begin{eqnarray*}
%		V(\S_{i})&=&V(\S_{j_{1}})+a_{ij_{1}}\\
%		V(\S_{j_{1}})&=&V(\S_{j_{2}})+a_{j_{1}j_{2}},\\
%	\end{eqnarray*}
%	hence, Lemma \ref{lem-aijk} yields 
%	$$V(\S_{i})=V(\S_{j_{2}})+a_{ij_{1}}+a_{j_{1}j_{2}}\le V(\S_{j_{2}})+a_{ij_{2}}.$$
%	Because $V(\S_{i})=\max_{j}\left(V(\S_{j})+a_{ij}\right)$ holds, we have 
%	$$V(\S_{i})= V(\S_{j_{2}})+a_{ij_{2}}, \text{and }a_{ij_{1}}+a_{j_{1}j_{2}}=a_{ij_{2}}.$$
%	Again, if $\S_{j_{2}}$ has maximal entropy, then we are done. Otherwise we continue the reasoning. Since there are finitely many components, 
%	\begin{itemize}
%		\item either we eventually find some $j_{k}$ satisfying the conditions, 
%		\item or we find a loop of $j_{l}$'s, say, $j_{l_{1}}, j_{l_{2}},\ldots j_{l_{r}}$, such that 
%		all the $\S_{j_{l}}$'s have non-maximal topological entropy and 
%		$$0=a_{l_{1}l_{2}}+a_{l_{2}l_{3}}+\ldots +a_{l_{r}l_{1}}\le a_{l_{1}l_{1}}<0,$$
%		which is a contradiction. 
%	\end{itemize}
%\end{proof}

%%%%%%%%
%%%%%%%%
\subsection{End of the proof of Theorem \ref{theoA}}

\begin{proposition}\label{gamma equal entropies}
	The limit $\disp\gamma = \lim_{\beta\to\infty} \frac{1}{\beta}\log(P(\beta A) - h )$ exists.  
Furthermore it is the unique eigenvalue for the $k\times k$ matrix in the Max-Plus formalism whose entries are the $a_{ij}$'s. 
		
\end{proposition}
		
	\begin{proof} Suppose that $\Omega = \S_1 \cup ...\cup \S_k \cup \S_{k+1}\cup...\cup \S_{L}$ where $h_{top}(\S_i) = h$ if $i\in\{1,...,k\}$ and $h_{top}(\S_i)<h$ if $i\in\{k+1,...,L\}$. 	{Let $\beta_j$ be a sequence such that $\lim_{\beta_j\to\infty} \frac{1}{\beta_j}\log(P(\beta_j A) - h )=\gamma$} and take a subsequence of this one such that $\frac{1}{\beta_{j_l}}\log(H_{\beta_{j_l}})$ converges for a function V. {From equation \eqref{eq - gamma and V}  $\gamma$ is a real number and we have  }
	\[ \gamma + V(\Sigma_i) = \max_{j\in\{1,...,L\}}V(\Sigma_j)+a_{ij},\,\, \forall \,i\in\{1,...,k\}.\]
	Applying Proposition \ref{prop-maxplusSiSJ} we get that such maximum is reached at some $j \in \{1,...,k\}$ and therefore 
	\[\gamma + V(\Sigma_i) = \max_{j\in\{1,...,k\}}V(\Sigma_j)+a_{ij},\,\,\forall\, i\in\{1,...,k\},\]
	which in Max-plus formalism can be rewritten as
\begin{equation}
	\label{eq-jajarere}
	\gamma \otimes \begin{bmatrix} V(\Sigma_1) \\ \vdots \\ V(\Sigma_k) \end{bmatrix}  = \begin{pmatrix} a_{11}&a_{12}&\dots&a_{1k}\\\vdots &\vdots &\vdots& \vdots \\ a_{k1}&a_{k2}&\dots&a_{kk}\end{pmatrix}\otimes \begin{bmatrix} V(\Sigma_1) \\ \vdots \\ V(\Sigma_k) \end{bmatrix} .
\end{equation}	
Observe that the matrix $M=(a_{ij})$ does not depend of the subsequence which defines $V$ and $\gamma$ and all its entries are real numbers.  Therefore, \cite[Th 3.23, p111]{BCOQ} yields that $M$ has a unique eigenvalue. Hence, the {function $\beta \to \disp\frac1\be\log(P(\be A)-h)$} admits a unique accumulation point, thus converges.
\end{proof}

%{\JM 
%\begin{remark} Consider the above proposition and proof. If $L=k$ and $V$ is a limit of $\frac{1}{\beta}\log(H_{\be A})$ as $\beta$ goes to $+\infty,$ then the numbers $V(\Sigma_1),...,V(\Sigma_k)$ are the entries of a max-plus eigenvector in above proof ({\bf the reciprocal could be false}). Furthermore, it follows from Lemmas \ref{mane} and \ref{V(sigma)} that for any $x\in X$ we have $$V(x) = \max_{j\in\{1,...,k\}} V(\Sigma_j) +S_j(x)$$ where $S_j$ is the Ma\~{n}\'{e} potential with respect to $\Sigma_j$. In this way, looking for eigenvectors in equation \eqref{eq-jajarere} may be helpful in order to find the possible calibrated subactions which are limits of $\frac{1}{\beta}\log(H_{\be A})$ as $\beta\to+\infty$.
%
%\end{remark}}

\begin{remark} 
In the case $L=k$, \ie all the irreducible components have maximal entropy, setting $\vec{V}:=(V(\S_{1}),\ldots, V(\S_{k}))$, $\vec{V}$ is an eigenvector for the matrix $M$. Nevertheless, uniqueness does not necessarily holds. It is thus meaning full to ask for which eigenvectors can be selected. 

\noindent
The question still have meaning if $k<L$, and in that case the role of components with small entropy is even more unpredictable. 

\end{remark}

\section{Proof of Theorem \ref{theoB}}\label{sec: theoB}
From now on, we consider {$X=X_{2}$  the full 2-shift}. $A$ is a {non-positive potential} of the form 
$$A|_{[01]}=b,\,\,A|_{[10]}=d,\,\,A(0^{\infty})=A(1^{\infty})=0,\,\,  A|_{[0^n1]} =a_n,\,\,A|_{[1^n0]} =c_n.$$
It is Lipschitz continuous which yields that the series $\sum a_{n}$ and $\sum c_{n}$ converge. 
{Furthermore, we remind that all numbers $b$, $d$ $a_{n}$ and $c_{n}$ are negative, which  yields  that the Aubry set is $\{0^{\8},1^{\8}\}$}.

In that case Theorem \ref{theoA} holds, and we set 
$$\ga:=\lim_{\be\to+\8}\dfrac1\be\log P(\be A),$$
as periodic orbits are subshift of finite type with zero entropy. 

We consider a family $(B_{\be})$ of potentials satisfying $|B_\beta|_\infty< e^{\beta \delta},\,\delta<\gamma$. { In the study of selection of subaction we consider also the hypothesis $Lip(B_\beta) \leq \beta c$ for some constant $c$ and any $\beta$.} 

We also remind that $H_{\phi}$ stand for the eigenfunction for the potential $\phi$. In the case $\phi=\be A$ the normalization is 
\begin{equation}\label{eq-normeigen}
\boxed{H_{\be A}(0^{\8})=1.}
\end{equation}

The section runs as follows:  we first give some technical results. Then we prove the convergence for $\disp\frac1\be\log P(\be A+B_{\be})$ { and $\disp\frac1\be\log H_{\be A+B_{\be}}$}. Finally we prove the convergence for the measure. 

%%%%%%%%
%%%%%%%
\subsection{Technical results}

%%%%%%%

%%%%%%
\subsubsection{Accumulation points for $\frac{1}{\beta}log(H_{\beta A + B_{\beta}})$}

We remind the double-inequality \eqref{eq-pressbb}
$$P(\be A)-|B_{\be}|_{\8}\le P(\be A+B_{\be})\le P(\be A)+|B_{\be}|_{\8}. $$
Hence, our assumption $|B_\beta|_\infty< e^{\beta \delta},\,\delta<\gamma$ immediately yields 

\begin{equation}\label{eq-cvpresspertu}
\lim_{\beta \to\infty} \frac{1}{\beta}\log(P(\beta A + B_\beta)) = \gamma.
\end{equation}

Next Lemma is a re-writing of the computation we did to get Equalities \eqref{ja1} and \eqref{ja2}.

\begin{lemma}\label{lemma2} { Suppose $Lip(B_\beta) \leq \beta c$ for some constant $c$ and any $\beta$ and that $|B_\beta|_\infty< e^{\beta \delta},\,\delta<\gamma$. }Any accumulation point $\wt V$ for $\frac{1}{\beta}\log(H_{\beta A+B_{\beta}})$  is a calibrated subaction of $A$ satisfying
	\[\gamma = A(10^{\infty}) + \tilde{V}(10^{\infty}) - \tilde{V}(0^{\infty})\]
	and
	\[\gamma = A(01^{\infty}) + \tilde{V}(01^{\infty}) - \tilde{V}(1^{\infty}).\]
\end{lemma}

\begin{proof} 
{ From Lemma \ref{existsubact} there exist convergent sub-sequences and any limit is a calibrated subaction for $A$. For any $x$ we have 
$$e^{P(\be A+B_{\be})}H_{\be A+B_{\be}}(x)=\sum_{y,\ \s(y)=x}e^{\be A(y)+B_{\be}(y)}H_{\be A+B_{\be}}(y).$$
In the special case $x=0^{\8}$, we get 
\[   (e^{P(\beta  A +B_\beta)} - e^{B_{\beta}(0^{\infty})})H_{\beta A+B_\beta}(0^{\infty}) =  e^{\beta A(10^{\infty})+B_\beta(10^\infty)}H_{\beta A +B_\beta}(10^{\infty}).\]
From \eqref{eq-cvpresspertu} and the hypothesis $|B_\beta|_\infty <e^{\beta \delta}$  we obtain 
$$\lim_{\beta\to+\8}\frac{1}{\beta}\log(e^{P(\beta  A +B_\beta)} - e^{B_{\beta}(0^{\infty})}) = \gamma.$$
Therefore
$$\ga+\wt V(0^{\8})=A(10^{\8})+\wt V(10^{\8}).$$
The proof of $\gamma = A(01^{\infty}) + \tilde{V}(01^{\infty}) - \tilde{V}(1^{\infty})$ is similar. }
\end{proof}  

%%%%%%%%
\subsubsection{A property for one $2\times 2$ Max-Plus Matrix}

\begin{proposition}
\label{prop-autoespaco}
Let $a,b,c,d$ be real numbers. Let $M$ be the matrix 
$$M:=\left(\begin{array}{cc}a+b+d & c+d \\a+b & b+c+d\end{array}\right).$$
Then its unique eigenvalue (for max-plus laws) is given by 
\[\lambda = \max\left\{a+b+d, b+c+d,\frac{a+b+c+d}{2}\right\}.\]
Furthermore, its eigen-space has dimension 1. 
\end{proposition}

\begin{proof}
Again,  \cite[Th 3.23, p111]{BCOQ} yields that $M$ has a unique eigen-value, say $\l$,  which is 
equal to $$\max\left\{a+b+d, b+c+d,\frac{a+b+c+d}{2}\right\}.$$
It thus remains to study the dimension of the eigenspace. 

	Equality
	\[\begin{pmatrix}a+b+d & c+d \\a+b & b+c+d\end{pmatrix} \otimes \begin{bmatrix} x\\y\end{bmatrix} = \lambda \otimes \begin{bmatrix}x\\y\end{bmatrix}\]
	can be rewritten as 
	
	\begin{subnumcases} {}
	\max\{a+b+d+x, \,\,c+d+y\} = \lambda+x \label{eq1toto}\\ 
	\max\{a+b+x,\,\,b+c+d+y\} = \lambda + y \label{eq2toto}
	\end{subnumcases}

To	prove that the eigen-space has dimension 1 means to prove that $y-x$ is constant. 
We consider three  cases. 

\medskip
$\bullet$ Case 1: If $\l=a+b+d$. Equation \eqref{eq1toto} yields $c+d+y\le a+b+d+x$. Hence 
$c+d+y< a+x$ and then $b+c+d+y < a+b+x$. 

Using this in \eqref{eq2toto} we get $ a+b+x=\lambda+y$ and then, as $\l=a+b+d$, we get $\boxed{y=x-d}$. 
%Furthermore, applying this conclusion in (Eq. 1) we obtain $a+b+d \geq c$. Therefore
%$\lambda =a+b+d =  \max\{a+b+d, b+c+d,\frac{a+b+c+d}{2}\}$.

\medskip
$\bullet$ Case 2: If $\l=b+c+d$.  The computation is the same than in the previous case, up to doing the exchanges
$$a\leftrightarrow c,\ b\leftrightarrow d,\ x\leftrightarrow y.$$
This yields $\boxed{x=y-b}$

\medskip
 $\bullet$ Case 3: If $\l=\disp\frac{a+b+c+d}{2}$.
 Equation \eqref{eq1toto}
 yields $c+d+y\le \dfrac{a+b+c+d}2+x$ which means 
$$y\le \dfrac{a+b-c-d}2+x.$$
Conversely, Equation \eqref{eq2toto} yields $a+b+x\le \dfrac{a+b+c+d}2+y$ which means 
$$y\ge \dfrac{a+b-c-d}2+x.$$
This finally yields $\boxed{y=x+\dfrac{a+b-c-d}2}$. 
% 
% 
% $\l\neq a+b+d$ and $\l\neq b+c+d$. Equations (Eq. 1) and (Eq. 2) will result in $c+d+y = \lambda + x$ and $a+b+x = \lambda + y$ respectively. Then we get $c+d+y-x = a+b+x-y$ and therefore $\boxed{y = x +\frac{a+b-c-d}{2}.}$
% Finally, we claim that in this case 
% \[\lambda =\frac{a+b+c+d}{2}= \max\{a+b+d, b+c+d,\frac{a+b+c+d}{2}\}.\]
%Indeed, returning to (Eq. 1) we get $c+d+x + \frac{a+b-c-d}{2} = \lambda +x$, which means $\lambda = \frac{a+b+c+d}{2}$. From (Eq. 1) we also get
%$a+b+d+x\leq \lambda+x = \frac{a+b+c+d}{2}+x$ and therefore $a+b+d \leq \frac{a+b+c+d}{2}$. Similarly, from (Eq. 2) we get
%$b+c+d+y \leq \frac{a+b+c+d}{2} + y$ and then $b+c+d \leq \frac{a+b+c+d}{2}$.
%

\end{proof}

\subsection{Proof of the convergence for $\disp\frac{1}{\be}\cdot log (H_{\be A+B_{\be}})$}
 Here we prove that perturbation does not affect the selection of the calibrated subaction. 
Note that Theorem \ref{theoA} holds in our case and we shall re-write Equation \eqref{eq-jajarere}. For that purpose we have to compute the Ma\~{n}\'{e} potential. 
For simplicity { we set $S_{i}(\cdot):=S(i^{\8},\cdot)$} with $i=0,1$.

\bigskip
We let the reader check that the functions $S_{i}$ are constant on the 2-cylinders $[01]$ and $[10]$ with values:
$$S_{0}(10)=b+\sum_{n\ge 2}a_{n},\ S_{0}(01)=\sum_{n\ge 2}a_{n},$$
$$S_{{1}}(10)=\sum_{n\ge 2}c_{n},\ S_{1}(01)=d+\sum_{n\ge 2}c_{n}.$$

Equation \eqref{eq-jajarere} yields that any accumulation point $V$ for $\dfrac1\be\log H_{\be A}$ satisfies the equation
    $$\gamma\otimes \begin{bmatrix}V(0^\8) \\V(1^\8)\end{bmatrix}=
  \left(\begin{array}{cc}A(10)+S_{0}(10) & A(10)+S_{1}(10) \\A(01)+S_{0}(01) & A(01)+S_{1}(01)\end{array}\right)\otimes \begin{bmatrix}V(0^\8) \\V(1^\8)\end{bmatrix},$$
 which can be rewritten as
 \begin{equation}
\label{eq2-maxpluswalters}
\gamma\otimes \begin{bmatrix}V(0^\8) \\V(1^\8)\end{bmatrix}=
  \left(\begin{array}{cc}d+b+\sum_{n\ge 2}a_n & d+\sum_{n\ge 2}c_n \\b+\sum_{n\ge 2}a_n & b+d+\sum_{n\ge 2}c_n\end{array}\right)\otimes \begin{bmatrix}V(0^\8) \\V(1^\8)\end{bmatrix}.
\end{equation}

%%%%%%%

Equation \eqref{eq2-maxpluswalters} can be rewritten as 
$$\gamma\otimes \begin{bmatrix} V(0^\8) \\V(1^\8)\end{bmatrix}=M\otimes \begin{bmatrix} V(0^\8) \\V(1^\8)\end{bmatrix},$$
where $M$ satisfies  Proposition \ref{prop-autoespaco}, with $a:=\sum_{n\geq 2} a_{n}$ and $c:=\sum_{n\geq 2} c_{n}$. This yields 
\begin{equation}\label{gamma Walters}
	\gamma=\max\left\{\sum_{n\geq 2} a_{n}+b+d, \,\,\sum_{n\geq 2} c_{n}+b+d,\,\,\frac{\sum_{n\geq 2} a_{n}+\sum_{n\geq 2} c_{n}+b+d}{2}\right\}.
\end{equation} 	

 Our condition $H_{\be A}(0^{\8})=1$ yields $V(0^{\8})=0$ and, as the eigenspace has dimension 1 (see Proposition \ref{prop-autoespaco}), this fixes the value for $V(1^{\8})$. As $\Omega = \{0^\infty,1^\infty\}$, applying Lemma \ref{mane}, the calibrated subaction $V$ is determined for any point $x\in \{0,1\}^{\mathbb{N}}$. 

\medskip
 Now we consider the potential $\beta A +B_\beta$. Lemma \ref{lemma2} shows that any accumulation point $\wt V$ for $\dfrac1\be\log(H_{\be A+B_{\be}})$ is a calibrated subaction for $A$ which satisfies

$$\gamma + \tilde{V}(0^{\infty})= A(10^{\infty}) + \tilde{V}(10^{\infty})$$
and 
$$\gamma + \tilde{V}(1^{\infty})= A(01^{\infty}) + \tilde{V}(01^{\infty}).$$

By Lemma \ref{mane} this can be rewritten as 
$$\gamma + \tilde{V}(0^{\infty})= A(10^{\infty}) + \max_{i\in\{1,2\}} S_{i^\infty}(10)+\tilde{V}(i^{\infty})$$
and
$$\gamma + \tilde{V}(1^{\infty})= A(01^{\infty}) + \max_{i\in\{1,2\}} S_{i^\infty}(01)+\tilde{V}(i^{\infty}).$$
Then, as in \eqref{eq2-maxpluswalters}, we have
$$\gamma\otimes \left(\begin{array}{c}\wt V(0^\8) \\\wt V(1^\8)\end{array}\right)=  \left(\begin{array}{cc}d+b+\sum_{n\ge 2}a_n & d+\sum_{n\ge 2}c_n \\b+\sum_{n\ge 2}a_n & b+d+\sum_{n\ge 2}c_n\end{array}\right)\otimes \left(\begin{array}{c}\wt V(0^\8) \\\wt V(1^\8)\end{array}\right).$$
Using the condition $H_{\be A+B_{\be}}(0^\infty)=1$ we have $\wt V(0^{\8})=0=V(0^{\8})$, therefore Prop. \ref{prop-autoespaco} shows that $\wt V(1^{\infty})=V(1^{\infty})$. Applying Lemma \ref{mane} we conclude that $\tilde{V} = V$.

%%%%%%%%%
%%%%%%%%
\subsection{Proof of the convergence for $\mu_{\be A+B_{\be}}$}

%%%%%%%%%%%
\subsubsection{The key condition to get convergence}

Now we present some auxiliary results for the study of selection of measure.
For any cylinder $[\om] \subset X$ let us denote by $\BBone_{[\om]}:X\to\{0,1\}$ the function satisfying
\[\BBone_{[\om]}(x) =\left\{\begin{array}{cc} 1& \text{if } x\in [\om] \\ 0 & \text{if } x\notin [\om]\end{array}\right. .\] 

{ \begin{lemma}\label{cilinder} { Let $A:X\to\mathbb{R}$ be any Lipschitz function such that $\mu_{\be A}$ converges to a probability $\mu_{\8}$.} 	
Let {$[\om]$} be a fixed cylinder and assume that for any $\wt\delta<\gamma$ the one parameter family of functions  $C_\beta:=e^{\be\wt\de}\BBone_{[\om]}$  satisfies  $$\limsup_{\beta \to +\infty}\mu_{\beta A + C_\beta}([\om]) \leq \mu_\infty([\om]) \text{ and } \liminf_{\beta \to +\infty}\mu_{\beta A - C_\beta}([\om]) \geq \mu_\infty([\om]).$$ 
	
	Then for any $\wt\delta'<\gamma$ and any Lipschitz continuous  functions $B_\beta$ satisfying $|B_\beta|_\infty < e^{\beta \wt\delta'}$,  $\lim_{\beta \to +\infty}\mu_{\beta A + B_\beta}([\om])=\mu_{\infty} ([\om])$ holds.
\end{lemma}

\begin{proof}
	Let us assume  $|B_\beta|_\infty <e^{\beta \wt\delta'}$ for  some $\wt\delta'<\gamma$ holds. Let $\wt\delta$ be a real number in $(\wt\delta',\gamma)$, let us fixe some { cylinder} $[\om]$. We set $C_\beta:=e^{\be\wt\de}\BBone_{[\om]}$. By definition of being equilibrium state, 
the next inequalities hold:
$$\int \beta A + C_{\beta} \,d\mu_{\beta A + C_\beta} + h({\mu_{\beta A + C_\beta}}) \geq 
	\int \beta A + C_{\beta} \,d\mu_{\beta A + B_\beta} + h({\mu_{\beta A + B_\beta}})$$
	and
	$$\int \beta A + B_{\beta} \,d\mu_{\beta A + C_\beta} + h({\mu_{\beta A + C_\beta}}) \leq 
	\int \beta A + B_{\beta} \,d\mu_{\beta A + B_\beta} + h({\mu_{\beta A + B_\beta}}).$$
	This yields
	\begin{equation}\label{eq1}
\int C_\beta - B_\beta \, d\mu_{\beta A + C_\beta} \geq \int C_\beta - B_\beta \, d\mu_{\beta A + B_\beta}.
\end{equation}
Similarly, inequalities 
$$\int \beta A + B_{\beta} \,d\mu_{\beta A + B_\beta} + h({\mu_{\beta A + B_\beta}}) \geq \int \beta A + B_{\beta} \,d\mu_{\beta A - C_\beta} + h({\mu_{\beta A - C_\beta}})
$$
and
$$ \int \beta A - C_{\beta} \,d\mu_{\beta A + B_\beta} + h({\mu_{\beta A + B_\beta}})\leq \int \beta A - C_{\beta} \,d\mu_{\beta A - C_\beta} + h({\mu_{\beta A - C_\beta}})$$
yield
\begin{equation}\label{eq2}
\int B_\beta +C_\beta \, d\mu_{\beta A + B_\beta} \geq \int B_\beta +C_\beta \, d\mu_{\beta A - C_\beta}.
\end{equation} 
Then we have 
\begin{eqnarray*}
\mu_{\beta A + C_\beta}([\om])&=&  \int \BBone_{[\om]} \, d\mu_{\beta A + C_\beta}\\
&=& e^{-\beta \wt\delta}\cdot\int C_{\be} \, d\mu_{\beta A + C_\beta}\\
&= &e^{-\beta \wt\delta}\cdot \int  C_{\be}-B_{\be}+B_{\be} \, d\mu_{\beta A + C_\beta}\\
\text{ by\eqref{eq1}} &\ge &e^{-\beta \wt\delta}\cdot\int C_{\be}-B_{\be}\, d\mu_{\beta A + B_\beta}+e^{-\beta \wt\delta}\cdot\int B_{\be} \, d\mu_{\beta A + C_\beta}\\
&\ge & \int \BBone_{[\om]}\, d\mu_{\beta A + B_\beta}-e^{-\beta \wt\delta}\cdot\int B_{\be}\, d\mu_{\beta A + B_\beta}+e^{-\beta \wt\delta}\cdot\int B_{\be} \, d\mu_{\beta A + C_\beta}\\
&\ge& \mu_{\be A+B_{\be}}([\om])-2e^{\be(\wt\de'-\wt\de)},
\end{eqnarray*}
where we use in the last inequality that $\mu_{\be A+B_{\be}}$ and $\mu_{\be A+C_{\be}}$ are probability measures. 
Since $\wt\de'<\wt\de$, letting $\be\to+\8$ yields 
$$\mu_{\8}([\om])\ge \limsup_{\be\to+\8}\mu_{\be A+C_{\be}}([\om])\ge \limsup_{\be\to+\8}\mu_{\be A+B_{\be}}([\om]).$$
Similarly, using \eqref{eq2}, { we also get:
\[	\mu_{\beta A - C_\beta}([\om])\leq   \mu_{\be A+B_{\be}}([\om])+ 2e^{\be(\wt\de'-\wt\de)}\]
and so}
$$\mu_{\8}([\om])\le \liminf_{\be\to+\8}\mu_{\be A-C_{\be}}([\om])\le \liminf_{\be\to+\8}\mu_{{ \be A+B_{\be}}}([\om]).$$

	\end{proof}}

\begin{corollary}  { Let $A:X\to\mathbb{R}$ be any Lipschitz function.} Assume that  $\mu_{\be A}$ converges to $\mu_{\8}$ and that for any fixed $n\in \mathbb{N}$, $\wt\delta<\gamma$ and any family of functions $C_\beta$ depending on $n$ coordinates and satisfying $|C_\beta|_\infty \leq e^{\beta \wt\delta}$,  $\mu_{\beta A + C_\beta}\to \mu_{\infty}$ holds.  Then, for any $\wt\delta'<\gamma$ and any Lipschitz functions $B_\beta$ satisfying $|B_\beta|_\infty < e^{\beta \wt\delta'}$ we have $\mu_{\beta A + B_\beta}\to\mu_\infty$. 
\end{corollary}

\begin{proof}
	The selection of measure is determined by the limits
	\[\lim_{\beta \to\infty} \mu_{\beta A +B_\beta}([\om]),\,\,\, [\om]\,\,\,\text{is\, cylinder}.\]
	For each fixed cylinder  $[\om]$, we can construct a function $C_\beta:=e^{\be\wt\de}\BBone_{[\om]},$ which depends of finite coordinates. By hipothesis, 
	\[\lim_{\beta \to \infty} \mu_{\beta A + C_\beta}([\om])=\lim_{\beta \to \infty} \mu_{\beta A - C_\beta}([\om])=\mu_\infty([\om]).\]  Then, from Lemma \ref{cilinder}  we have 
	\[\lim_{\beta \to \infty} \mu_{\beta A + B_\beta}([\om]) =\mu_\infty([\om]).\]
\end{proof}

\subsubsection{Selection of maximizing measure}

 Now we prove the convergence of the measures  $\mu_{\beta A +B_\beta}$ and $\mu_{\be A}$ as stated in Theorem \ref{theoB}.
We will present explicitly the limit measure.

By Lemma \ref{pressureselect}, any accumulation point for $\mu_{\be A+B_{\be}}$ is $A$-maximizing, hence of the form $s\delta_{0^\infty} + (1-s) \delta_{1^\infty}$, with  $s\in [0,1]$. Note that it is sufficient to get that $\mu_{\beta A+B_\beta}([0])$ and $\mu_{\beta A+B_\beta}([1])$ converge to determine $s$. 
As $\mu_{\beta A+B_\beta}([1]) = 1-\mu_{\beta A+B_\beta}([0])$, it thus sufficient to get that $\mu_{\beta A+B_\beta}([0])$ converges. 
Furthermore, Lemma \ref{cilinder} yields that the convergence for $\mu_{\beta A+B_\beta}([0])$  follows from the convergence of $\mu_{\beta A+C_{\be}\BBone_{[0]}}([0])$ with $|C_{\be}|$ going faster to 0 than $e^{\be\ga}$. 
Therefore we can suppose that $B_\beta$ only depends  on the first coordinate and satisfies 
\[B_\beta(0)=a_\beta \,\,\,\text{and}\,\,\,B_\beta(1) = 0,\] 
with $|a_\beta|\leq e^{\beta \delta}$,\,$\delta<\gamma$.

Let us set 
$$S_{0}(\beta) = \frac{ 1+\sum_{j=1}^{\infty}(j+1)e^{\beta(a_{2}+...+a_{1+j})+ja_\beta-jP(\beta A+B_\beta)}}{1+\sum_{j=1}^{\infty}e^{\beta(a_{2}+...+a_{1+j})+ja_\beta-jP(\beta A+B_\beta)}},$$
$$S_{1}(\beta) = \frac{1+\sum_{j=1}^{\infty}(j+1)e^{\beta(c_{2}+...+c_{1+j})-jP(\beta A+B_\beta)}}{ 1+\sum_{j=1}^{\infty}e^{\beta(c_{2}+...+c_{1+j})-jP(\beta A+B_\beta)}}. $$
{ The potential $\beta A +B_\beta$ does not satisfy $\eqref{W}$ because it can be positive in a neighborhood of $0^{\infty}$ while it is equal to zero in $1^{\infty}$, but the same proof of Proposition 6 in \cite{BLM} can be applied in order to conclude that }
\[\mu_{\beta A+B_\beta}([0]) = \frac{S_0(\beta)}{S_0(\beta) + S_1(\beta)} \hspace{1cm} \text{and} \hspace{1cm} \mu_{\beta A+B_\beta}([1]) = \frac{S_1(\beta)}{S_0(\beta) + S_1(\beta)}.\]

In a first step we  prove that the limit for $\frac{S_0(\beta)}{S_1(\beta)}=\frac{\mu_{\beta A+B_\beta}([0])}{\mu_{\beta A+B_\beta}([1])}$ does not depend on $a_\beta$. In a second moment we prove that such limit exists.  
For that purpose, we consider a series of claims. From now on in this section the notation $f(\beta)\sim g(\beta)$ means $\lim_{\beta \to +\infty} \frac{f(\beta)}{g(\beta)} = 1.$ Clearly if $S_0(\beta)\sim F(\beta)$ and $S_1(\beta) \sim G(\beta)$ then $\lim_{\beta \to +\infty} \frac{S_0(\beta)}{S_1(\beta)} = \lim_{\beta \to +\infty} \frac{F(\beta)}{G(\beta)}.$ Next lemma present basic properties of $\sim$ which will be used. % We also denote by $[\be]$ be the integer part for $\beta$.

\begin{lemma}\label{lema:basic sim} For any positive functions $f, g$ we have:\newline
	1. If $\lim_{\beta\to+\infty} f(\beta) =0$ then $(1+f(\beta)+g(\beta))\sim (1+g(\beta))$;\newline
	2. If $\lim_{\beta\to+\infty} f(\beta) =1$ then $(1+f(\beta)\cdot g(\beta)) \sim (1+g(\beta))$; \newline
	3. If $\lim_{\beta\to+\infty} f(\beta) =1$ and $h$ is a function satisfying $ (1+g(\beta)) \leq 1+h(\beta) \leq 1+f(\beta)\cdot g(\beta)$ then $(1+h(\beta))\sim (1+g(\beta))$;\newline
	4. If $f(\beta) \sim {f_0}(\beta) $ then $(1+f(\beta)g(\beta))\sim (1+{f_0}(\beta)g(\beta))$.
\end{lemma}
\begin{proof} 
	The proof of item 1. follows from
	\[\lim_{\beta \to+\infty} \frac{1+f(\beta)+g(\beta)}{1+g(\beta)} = 1+\lim_{\beta \to+\infty} \frac{f(\beta)}{1+g(\beta)}=1,\]
	because $f(\beta)\to 0$ and $g$ is positive.
	
	The proof of item 2. follows from
	\[\lim_{\beta \to+\infty} \frac{1+f(\beta)\cdot g(\beta)}{1+g(\beta)} = 1 + \lim_{\beta \to+\infty} \frac{(f(\beta)-1)\cdot g(\beta)}{1+g(\beta)} = 1,\]
	because $(f(\beta) -1) \to 0 $ and $0<\frac{g(\beta)}{1+g(\beta)}<1.$
	
	The proof of item 3. is a consequence of 2.. Indeed, we have 
	\[1 = \frac{1+g(\beta)}{1+g(\beta)} \leq \frac{1+h(\beta)}{1+g(\beta)} \leq \frac{1+f(\beta)g(\beta)}{1+g(\beta)}\]
	and the right hand side function converges to 1 (because item 2.). 
	
	The proof of item 4. is again a consequence of item 2..  We have
	\[1+f(\beta)g(\beta) = 1 + \frac{f(\beta)}{{f_0}(\beta)}{f_0}(\beta)g(\beta) \sim 1 + {f_0}(\beta)g(\beta)\]
because, by hypothesis, $\frac{f(\beta)}{{f_0}(\beta)}\to 1$.

\end{proof}

\begin{claim}\label{claim1}
 \begin{equation}\label{eq - 0}
\lim_{\beta \to +\infty} \sum_{j<\beta} (j+1)e^{\beta(a_{2}+...+a_{1+j})+ja_\beta-jP(\beta A+B_\beta)}=0.
\end{equation}
\end{claim}
\begin{proof}[Proof of the Claim]
 Note that $P(\beta A+B_\beta)\geq P(\beta A)-e^{\beta \delta} >0$  holds for $\beta$ large enough.
 Hence, 
 $$0\leq \sum_{j<\beta}(j+1)e^{\beta(a_{2}+...+a_{1+j})+ja_\beta-jP(\beta A+B_\beta)}\leq   (\beta+1)^2 e^{\beta\cdot a_{2} +\beta \cdot |a_\beta|}.$$
As $a_\beta \to 0$ and $a_2<0$ we conclude \eqref{eq - 0}.
\end{proof}

 \begin{remark}
\label{rem-a2}
We emphasize that this is the unique place where we assume $a_{2}<0$ and $c_{2}<0$. Except here, assumptions $\sum a_{n}\le 0$, $\sum c_{n}\le 0$, $b+d<0$ + the fact that the Aubry set is $\{0^{\8},1^{\8}\}$  are sufficient. 
$\blacksquare$\end{remark}

It follows from Claim \ref{claim1} and item 1. of Lemma \ref{lema:basic sim} that
\begin{equation}\label{eq:S_0 first}
S_0(\beta) \sim  \frac{ 1+\sum_{j\geq \beta}(j+1)e^{\beta(a_{2}+...+a_{1+j})+ja_\beta-jP(\beta A+B_\beta)}}{1+\sum_{j\geq \beta}e^{\beta(a_{2}+...+a_{1+j})+ja_\beta-jP(\beta A+B_\beta)}}.
\end{equation} 

In the following $a$ stands for $\sum a_{n}$ and $c$ stands for $\sum c_{n}$.
	
\begin{claim}\label{claim4} There exists $C>0$ such that for every $j\ge 1$,
	$$c\le c_2+...+c_{1+j} \le c+ C\theta^{j}
	\text{ and }
	a\le a_2+...+a_{1+j} \le a+C\theta^{j}$$
\end{claim}
\begin{proof}[Proof of the Claim]
	Indeed, as $A$ is Lipschitz, there exists $C_0>0$ such that, for $n\geq 2$, 
	$$-c_n = -A(1^{n}0^\infty) \leq -A(1^{\infty}) + C_0\cdot d(1^n0^\infty,1^{\infty}) = C_0\theta^n.$$
	Then, for each $j \geq \beta$, we have 
	$$c_2+...+c_{1+j} = \sum_{i=2}^{\infty}c_{i}- \sum_{i=j+2}^{\infty}c_{i}\leq \sum_{i=2}^{\infty} c_{i}+\sum_{i=j+2}^{\infty}C_0\theta^{i}\leq \sum_{i=2}^{\infty}c_{i} + \frac{C_0\theta^{j+2}}{1-\theta}. $$
	A similar  computation can be used for $a$'s.
\end{proof}

It follows from Claim \ref{claim4} that   
\[1+ e^{\beta a}\sum_{j\geq \beta}e^{-jP(\beta A+B_\beta)}\leq 1+\sum_{j\geq \beta}e^{\beta(a_{2}+...+a_{1+j})-jP(\beta A+B_\beta)} \leq 1+ e^{\beta C\theta^{\beta}}\cdot e^{\beta a}\sum_{j\geq \beta}e^{-jP(\beta A+B_\beta)}.\]
As $\lim_{\beta\to+\infty}e^{\beta C\theta^{\beta}}= 1$, applying item 3. of Lemma \ref{lema:basic sim}, we get 
\[ 1+\sum_{j\geq \beta}e^{\beta(a_{2}+...+a_{1+j})-jP(\beta A+B_\beta)} \sim 1+ e^{\beta a}\sum_{j\geq \beta}e^{-jP(\beta A+B_\beta)}.\]
With a similar computation we also obtain
\[  1+\sum_{j\geq \beta}(j+1)e^{\beta(a_{2}+...+a_{1+j})+ja_\beta-jP(\beta A+B_\beta)} \sim 1+e^{\beta a} \sum_{j\geq \beta}(j+1)e^{ja_\beta-jP(\beta A+B_\beta)}.\]
Applying this results in equation \eqref{eq:S_0 first} we get
\begin{equation}\label{eq:S_0 second} 
	S_0(\beta) \sim \frac{ 1+e^{\beta a}\sum_{j\geq \beta}(j+1)e^{ja_\beta-jP(\beta A+B_\beta)}}{1+e^{\beta a}\sum_{j\geq \beta}e^{ja_\beta-jP(\beta A+B_\beta)}}.
\end{equation}

Let us consider now the next step. From simple computations we get, for any $z>0$ and $k \in\mathbb{N}$,
\[\sum_{j=k}^{\infty} e^{-jz} = \dfrac{e^{-kz}}{1- e^{-z}} \,\,\,\,\,\,and \,\,\,\,\,\,\sum_{j=k}^{\infty} (j+1) e^{-jz} = \frac{e^{-kz}}{1-e^{-z}}\left(k+\frac1{1-e^{-z}}\right) .\]

\begin{claim}\label{claim3}
	\[\sum_{j\geq \beta}e^{ja_\beta-jP(\beta A+B_\beta)} \sim \frac{1}{P(\beta A)}\,\,\, and \,\,\, \sum_{j\geq \beta}(j+1)e^{ja_\beta-jP(\beta A+B_\beta)} \sim \frac{1}{(P(\beta A))^2} .\]
\end{claim}
\begin{proof}[Proof of the Claim]
As $\lim_{x\to+\infty}\frac{x}{1-e^{-x}}=1$, if $z$ is a function such that $z(\beta)\downarrow 0$ and $\beta\cdot z(\beta) \to 0$  we get
\[	\sum_{j\geq \beta} e^{-jz(\beta)} = \dfrac{e^{-\beta z(\beta)}}{1- e^{-z(\beta)}} \sim \frac{1}{z(\beta)}.\]
Furthermore,
\[\sum_{j\geq \beta} (j+1) e^{-jz(\beta)} = \frac{e^{-\beta z(\beta)}}{1-e^{-z(\beta)}}\left(\beta+\frac1{1-e^{-z(\beta)}}\right) \sim \frac{\beta}{z(\beta)}+\frac{1}{(z(\beta))^2} \sim \frac{\beta z(\beta)+1}{(z(\beta))^2}\sim \frac{1}{(z(\beta))^2}.\]
Finally we take $z(\beta) = -a_\beta +P(\beta A+B_\beta)$. Observe that $z(\beta) \sim P(\beta A)$ and $\beta\cdot P(\beta A) \to 0$ because $P(\beta A)$ { and $a_{\be}$ converge exponentially fast  to zero, respectively at rate $e^{\be.\ga}$ and $e^{\be.\de}$ with $\gamma$ as given in \eqref{gamma Walters} and $\de<\ga$}. 
\end{proof}

Claim \ref{claim3}, item 4. of Lemma \ref{lema:basic sim} and equation \eqref{eq:S_0 second} together yield
\[S_0(\beta) \sim  \frac{ 1+e^{\beta a}\frac{1}{(P(\beta A)^2}}{1+e^{\beta a}\frac{1}{P(\beta A)}}.\]
Similarly we get 
\[S_1(\beta) \sim  	\frac{ 1+e^{\beta c}\frac{1}{(P(\beta A))^2}}{1+e^{\beta c}\frac{1}{P(\beta A)}}\]
and therefore
	\begin{equation}
		\label{eq-fina3}
		\frac{S_{0}(\be)}{S_{1}(\be)}\sim \dfrac{(P(\be A))^2+e^{\be a}}{P(\be A)+e^{\be a}}\dfrac{P(\be A)+e^{\be c}}{(P(\be A))^2+e^{\be c}}.
	\end{equation}

 Observe that the right hand side does not depend of $B_\beta$. Then the limit of the ratio $\disp\dfrac{\mu_{\be A +B_\be}([0])}{\mu_{\be A+B_\be}([1])}=\frac{S_{0}(\be)}{S_{1}(\be)}$  does not depend of $B_\beta$. This finally shows that $\mu_{\be A+B_{\be}}$ converges to the same limit than  $\mu_{\be A}$, if it converges.  
 
 \bigskip
 Next proposition shows that $\mu_{\beta A+B_\beta}$ actually converges and  explicitly gives the limit.

\begin{proposition}\label{prop:limits measure} Under above hypothesis concerning $B_\beta$ and $A$: \newline
	1. If $a=c$ then $\mu_{\beta A +B_\beta} \to \frac{\delta_{0^{\infty}}+\delta_{1^{\infty}}}{2}$. \newline
	2. If $a>c$ and $a+b+d<c$ then $\mu_{\beta A +B_\beta} \to \frac{\delta_{0^{\infty}}+\delta_{1^{\infty}}}{2}$.\newline
	3. If $a>c$ and $a+b+d>c$ then $\mu_{\beta A +B_\beta} \to \delta_{0^{\infty}}$\newline
	4. If $a>c$ and $a+b+d=c$ then $\mu_{\beta A +B_\beta} \to \frac{10+2\sqrt{5}}{20}\delta_{0^{\infty}}+\frac{10-2\sqrt{5}}{20}\delta_{1^{\infty}}.$ \newline 
We get symmetric results in the case $c>a$.
\end{proposition} 
\begin{proof} 
We will consider \eqref{eq-fina3}. Furthermore, from 	\eqref{gamma Walters} we have
\[\gamma=\max\left\{a+b+d, \,\,c+b+d,\,\,\frac{a+c+b+d}{2}\right\}.\]	

The proof of item 1. is consequence of \eqref{eq-fina3}. {In the case $a=c$, the right-hand side term in \eqref{eq-fina3} equals 1, which yields that $\frac{S_{0}(\be)}{S_{1}(\be)}$ goes to 1 as $\be\to+\8$.}

\bigskip
For the proof of other cases we consider $l(\beta) := P(\beta A)/e^{\beta \gamma}$. Then $P(\beta A) = l(\beta) e^{\beta \gamma}$ and $\lim_{\be\to+\8}\frac{1}{\beta}\log(l(\beta)) =0$. It follows that for any constant $\epsilon>0$ we have $e^{-\beta \epsilon} <l(\beta) <e^{\beta \epsilon}$ for sufficiently large $\beta$. With such notation, from \eqref{eq-fina3} we have
\[S_0(\beta)/S_1(\beta)\sim \dfrac{l^2(\beta)e^{2\beta \gamma}+e^{\be a}}{l(\beta)e^{\beta\gamma}+e^{\be a}}\dfrac{l(\beta)e^{\beta\gamma}+e^{\be c}}{l^2(\beta)e^{2\beta \gamma}+e^{\be c}}.\]

Proof of item 2: If $a>c$ and $a+b+d<c$ then
$\frac{a+b+d+c}{2} > a+b+d$ and so $\gamma = \frac{a+b+d+c}{2}$. We have also
\[a>c\Rightarrow a>c+b+d \Rightarrow \frac{a}{2} > \frac{c+b+d}{2} \Rightarrow a > \frac{a}{2}+\frac{c+b+d}{2} \Rightarrow a>\gamma\]
and similarly
\[c>a+ b+d \Rightarrow c>\gamma.\]
As $a>\gamma$ and $c>\gamma$ it follows  that 
$$S_0(\beta)/S_1(\beta)\sim \dfrac{l^2(\beta)e^{2\beta \gamma}+e^{\be a}}{l(\beta)e^{\beta\gamma}+e^{\be a}}\dfrac{l(\beta)e^{\beta\gamma}+e^{\be c}}{l^2(\beta)e^{2\beta \gamma}+e^{\be c}} \to 1$$ and then $\mu_{\beta A +B_\beta} \to \frac{\delta_{0^{\infty}}+\delta_{1^{\infty}}}{2}$.

Proof of item 3: If $a+b+d>c$ then $\frac{a+b+d+c}{2} < a+b+d$ and then $\gamma = a+b+d$. It follows that $c<\gamma<a$. Then, using any auxiliary number $\delta$ such that $c<\delta<\gamma$ and $2\gamma <\delta$ we have
\[S_0(\beta)/S_1(\beta)\sim \dfrac{l^2(\beta)e^{2\beta \gamma}+e^{\be a}}{l(\beta)e^{\beta\gamma}+e^{\be a}}\dfrac{l(\beta)e^{\beta\gamma}+e^{\be c}}{l^2(\beta)e^{2\beta \gamma}+e^{\be c}} \sim 1\cdot \frac{l(\beta)e^{\beta(\gamma-\delta)}+e^{\beta(c-\delta)}}{l^2(\beta)e^{\beta (2\gamma-\delta)}+e^{\beta (c-\delta)}} \to +\infty.\]

Proof of item 4:  As $a+b+d=c$ we have $a+b+d = \frac{a+b+d+c}{2}$ and then $\gamma  =a+b+d = \frac{a+b+d+c}{2}$. In this case $\gamma <a$ and $\gamma =c$. Then we have
$$S_0(\beta)/S_1(\beta)\sim 1\cdot \dfrac{l(\beta)e^{\beta c}+e^{\be c}}{l^2(\beta)e^{2\beta c}+e^{\be c}} \sim \dfrac{l(\beta)+1}{l^2(\beta)e^{\beta c}+1} \sim l(\beta)+1. $$ 
It remains to study the limit for $l(\beta)$.
From \cite{Walters} corollary 3.5 we have (here in the present paper $b_j =b$ and $d_j=d$ for any $j$) 
\[e^{2P(\beta A)}=e^{\beta (b+d)}\left(1+\sum_{j\geq 1}e^{\beta(a_2+...+a_{1+j})-jP(\beta A)}\right)\left(1+\sum_{j\geq 1}e^{\beta(c_2+...+c_{1+j})-jP(\beta A)}\right) .\]
Applying similar computations as above (in order to get \eqref{eq-fina3}) we obtain
\[ e^{2P(\beta A)}\sim e^{\beta (b+d)}\left(1+\frac{e^{\beta a}}{P(\beta A)}\right)\left(1+\frac{e^{\beta c}}{P(\beta A)}\right).\]
Then
\[e^{2P(\beta A)}\sim e^{\beta (b+d)} + \frac{e^{\beta(b+d+c)}}{P(\beta A)} + \frac{e^{\beta(b+d+a)}}{P(\beta A)} + \frac{e^{\beta(b+d+a+c)}}{(P(\beta A))^2}\]
and
\[e^{2P(\beta A)}\sim e^{\beta (b+d)} + \frac{e^{\beta(b+d+c)}}{l(\beta)e^{\beta \gamma}} + \frac{e^{\beta(b+d+a)}}{l(\beta)e^{\beta \gamma}} + \frac{e^{\beta(b+d+a+c)}}{l^2(\beta)e^{2\beta\gamma}}.\]
As $e^{2P(\beta A)}\to 1$ and from the hyphothesis of item 4. we have $\gamma = c = a+b+d = \frac{a+b+d+c}{2}$, we get
\[1\sim  \frac{1}{l(\beta)} + \frac{1}{l^2(\beta)}.\]
As $l(\beta)$ is positive we obtain $l(\beta) {\RL\to} \frac{1+\sqrt{5}}{2}.$ Finally we conclude that $S_0(\beta)/S_1(\beta)\to  \frac{3+\sqrt{5}}{2}$. Consequently $\mu_{\beta A}([0]) \to \frac{3+\sqrt{5}}{5+\sqrt{5}} = \frac{10+2\sqrt{5}}{20}$ and therefore  	$\mu_{\beta A}([1]) \to \frac{2}{5+\sqrt{5}} = \frac{10-2\sqrt{5}}{20}$.
	\end{proof} 	

\begin{remark}
 It is noteworthy that knowing $A$ close to $0^{\8}$ and $1^{\8}$ is not sufficient to determine the selection. This has been already pointed out in \cite{BLM}. Indeed, $b+d$ plays a role to determine the limit. 

However, we point ou that in cases 1 and 2 in Proposition \label{prop:limits measure}, $\ga$ is given by the length-2 cycle $0^{\8}\to 1^{\8}\to 0^{\8}$ in the matrix $M$, and then $\mu_{\8}$ is $\frac12(\de_{0^{\8}}+\de_{1^{\8}})$. In case 3, $\ga$ is given by the length-1 cycle $0^{\8}\to 0^{\8}$ and the limit measure is $\de_{0^{\8}}$. Finally in case 4, $\ga$ is given both by length-1 cycle $0^{\8}\to 0^{\8}$ and length-2 cycle $0^{\8}\to 1^{\8}\to 0^{\8}$, and the limit measure is an asymmetric combination of  $\de_{0^{\8}}$ and $\de_{1^{\8}}$.

In all these cases, $a\ge c$ and $\mu_{\8}$ contains some positive part of $\de_{0^{\8}}$. This confirm that flatness may also be a criteria for selection (see \cite{Leplaideur2}).

\end{remark}

\appendix
{

\section{ About the hypothesis  in Theorem \ref{theoB}} \label{initial}

Consider a Lipschitz function $A:X \to \mathbb{R}$ satisfying $m(A) = 0$ and suppose there exists the limit $\gamma = \lim_{\beta\to+\infty}\frac{1}{\beta}\log(P(\beta A) - h)$, where $h=\sup_{\mu\in \mathcal{M}_{max}(A)}h_{\mu}(\sigma)$. Considering the hypothesis in Theorem \ref{theoB} concerning the perturbation, it is natural to ask what happens if $B_\beta$ converges to zero slower than $e^{\beta \gamma}$. Next example present a discussion in this way and we will show a change in the selection of subaction and measure. We assume, in this example, that the eigenfuntions $H_\beta$ are normalized by $H_\beta(0^{\infty})=1$.

\begin{example}\label{example main}
	
	In this example we suppose $X=\{0,1\}^{\mathbb{N}}$ and consider two numbers $\gamma$ and $\eta$ satisfying $\gamma<\eta<0$. Let $A,B_\beta:\{0,1\}^\mathbb{N} \to \mathbb{R}$ be functions which are constant in each 2-cylinder and defined by  
\[		A(0,1)=A(1,0)=\gamma,\,\ A(0,0) =A(1,1)=0 , \]
\[ B_\beta(0,0) = B_\beta(1,0) = B_\beta(0,1)=0 \,\,\,and \,\,\, B_\beta(1,1) = \log (1+e^{\beta \eta}).\]

\noindent
 We claim that, as $\beta$ goes to $+\infty$:\newline
 1.   $\frac{1}{\beta}\log(P(\beta A)) \to \gamma$;\newline
 2. $\frac{1}{\beta}\log |B_{\beta}|_\infty = \eta $ (then $|B_\beta|_\infty$ seems $e^{\beta \eta}$ with $\eta>\gamma$ );\newline
 3.   $\frac{1}{\beta}\log(H_{\beta A})(1^\infty) \to 0$; \newline 
 4. $\frac{1}{\beta}\log(H_{\beta A+B_\beta})(1^\infty) \to \eta -\gamma $; \newline 
 5. $\mu_{\beta A} \to \frac{1}{2}(\delta_{0^\infty}+\delta_{1^\infty})$; \newline
 6. $\mu_{\beta A+B_{\beta}} \to \delta_{1^\infty}$. \newline
 From items 1-6 clearly the perturbation $B_\beta$ converges to zero exponentially fast, but slower than $e^{\beta \gamma}$, and it changes the results concerning selections of subaction and measure. 

Proof of 1.: {In this case note that Theorem \ref{theoA} holds. 
The matrix $M$ in the Prop. \ref{prop-autoespaco} is then equal to 
$$M=\left(\begin{array}{cc}2\ga & \ga \\\ga & 2\ga\end{array}\right),$$
and its unique eigenvalue for the Max-Plus formalism is $\ga$. 
As entropy for $0^{\8}$ and $1^{\8}$ is zero we get 1. }

%the action of Ruelle operator $L_{\beta A}$ over functions $\psi$ which depend just of one coordinate can be written using matrix notation as
%	\[(L_{\beta A}(\psi)(0)\,\,\,\,\,\,\,L_{\beta A}(\psi)(1))=(\psi(0) \,\,\,\,\,\,\, \psi(1))\cdot \begin{pmatrix} e^{\beta A(0,0)}& e^{\beta A(0,1)}\\e^{\beta A(1,0)}&e^{\beta A(1,1)}\end{pmatrix}.\]
%	Consider this matrix $$\left(\begin{array}{cc} e^{\beta A(0,0)} & e^{\beta A(0,1)}\\ e^{\beta A(1,0)} & e^{\beta A(1,1)}\end{array}\right) =  \left(\begin{array}{cc} 1 & e^{\beta \gamma}\\ e^{\beta \gamma} & 1\end{array}\right).$$
%	The main eigenvalue for such matrix is given by
%	\[\lambda_{\beta}=1+e^{\beta \gamma}.\]
%	As $\log(1+x)= x-\frac{x^2}{2}+\frac{x^{3}}{3}-\frac{x^{4}}{4}+...$ for $x$ near to $0$, we have $\log(1+x) = x[1+o(x)]$ with $o(x)\to 0$ as $x\to 0$. Then
%	$$P(\beta A) = \log(\lambda_\beta) =   e^{\beta \gamma}[1+o(e^{\beta \gamma})]$$ and consequently $\frac{1}{\beta}\log(P(\beta A)) \to \gamma$, as claimed. 

\medskip	
Proof of 3.:	the eigenfunction associated to $\lambda_{\beta A}$ depends just of one coordinate and satisfies
	\[\lambda_\beta  H_{\beta A} (0)= L_{\beta A}(H_{\beta A})(0) =1\cdot H_{\beta A} (0) + e^{\beta \gamma}H_{\beta A} (1)  .\]
	Hence, as by convention $H_{\beta A}(0) =1$, we get
	\begin{align*}
		H_{\beta A}(1) &= \frac{\lambda_\beta - 1}{e^{\beta \gamma}},
	\end{align*}
	 with $\l_{\be}=e^{P(\be)}$. Doing $\frac1\be\log$ in this last equality, and $\be\to+\8$ we get 
	 $V(1^{\infty})=\lim_{\beta\to\infty}\frac{1}{\beta}\log(H_{\beta A}(1)) = 0$.
	
\medskip
Proof of 5.: {as $A$ is locally constant we know that $\mu_{\be A}$ converges (see \cite{Bremont, Leplaideur}). Furthermore, symmetry yields $\mu_{\beta A}  \to \frac{\delta_{0^{\infty}}+\delta_{1^{\infty}}}{2}$.}	
	
\medskip
Proof of 2.: as $\log(1+x) = x[1+o(x)]$, for $x$ near to $0$, we get \[\lim_{\beta\to\infty}\frac{1}{\beta}\log(|B_\beta|_\infty) = \eta.\]
	
\medskip
Proof of 4. the associated matrix for the Ruelle operator $L_{\beta A + B_\beta} $ is given by 
	$$\left(\begin{array}{cc} 1 & e^{\beta \gamma}\\ e^{\beta \gamma} & 1+e^{\beta \eta}\end{array}\right) $$
	{with dominating eigenvalue} 
	\begin{equation}\label{eq - 0ex}  \tilde{\lambda}_\beta = 1+\frac{e^{\beta \eta} +\sqrt{ e^{2\beta \eta} + 4e^{2\beta \gamma}}}{2}.
	\end{equation} 
	The associated main eigenfunction  depends just of one coordinate and satisfies
	\[\tilde{\lambda}_\beta  H_{\beta A + B_\beta} (0)=H_{\beta A + B_\beta} (0) + e^{\beta \gamma }H_{\beta A + B_\beta} (1).\]
	Hence, using again that  $H_{\beta A + B_\beta} (0)=1$ by convention, we get
	\begin{equation}\label{eq - 1ex}
		H_{\beta A+B_\beta}(1) = \frac{\tilde{\lambda}_\beta - 1}{e^{\beta \gamma}}=  \displaystyle{\frac{e^{\beta \eta} +\sqrt{ e^{2\beta \eta} + 4e^{2\beta \gamma}}}{2e^{\beta \gamma}}}.
	\end{equation}
	It follows that
	\[\lim_{\beta\to\infty} \frac{1}{\beta}\log( H_{\beta A + B_\beta}(1) )= \lim_{\beta \to\infty}\frac{1}{\beta}\log \left(\frac{e^{\beta \eta} +\sqrt{ e^{2\beta \eta} + 4e^{2\beta \gamma}}}{2e^{\beta \gamma}}\right)=\eta-\gamma>0.\]

Proof of 6.: consider the normalized function $$\overline{\beta A + B_\beta}:= \beta A + B_\beta + \log(H_{\beta A+B_\beta}) - \log(H_{\beta A+B_{\beta}} \circ \sigma) - \log(\tilde{\lambda}_\beta).$$ The associated Ruelle operator is identified with the matrix 
	\[\left( \begin{array}{ccc} 
		\frac{1}{\tilde{\lambda}_\beta} & &\frac{e^{\beta \gamma}}{H_{\beta A+B_\beta}(1)\tilde{\lambda}_\beta} \\ & & \\
		\frac{e^{\beta \gamma}H_{\beta A+B_\beta}(1)}{\tilde{\lambda}_\beta} & &\frac{1+e^{\beta \eta}}{\tilde{\lambda}_\beta} \end{array}\right).\]
	
	In order to study the selection of $\mu_\beta = \mu_{\beta A + B_\beta}$ it is sufficiently to study $p_0^\beta=\mu_\beta ([0])$ and $p_1^\beta = \mu_\beta([1])$. As $p_0^\beta$ and $p_1^\beta$ satisfy
	\[ \left( \begin{array}{ccc} 
		\frac{1}{\tilde{\lambda}_\beta} & &\frac{e^{\beta \gamma}}{H_{\beta A+B_\beta}(1)\tilde{\lambda}_\beta} \\ & & \\
		\frac{e^{\beta \gamma}H_{\beta A+B_\beta}(1)}{\tilde{\lambda}_\beta} & &\frac{1+e^{\beta \eta}}{\tilde{\lambda}_\beta} \end{array}\right) \begin{bmatrix} p_0^\beta \\ \\ p_1^\beta \end{bmatrix} =  \begin{bmatrix} p_0^\beta \\ \\ p_1^\beta \end{bmatrix}\]
	we have
	\[p_0^\beta e^{\beta \gamma}H_{\beta A+B_\beta}(1) + p_1^\beta (1+e^{\beta \eta}) = \tilde{\lambda}_\beta p_1^\beta\]
	and so, using that $p_1^\beta = 1-p_0^\beta$, we get
	\[p_0^\beta = \frac{\tilde{\lambda}_\beta - 1-e^{\beta \eta}}{e^{\beta\gamma}H_{\beta A+B_\beta}(1)+\tilde{\lambda}_\beta - 1-e^{\beta\eta}}\stackrel{\eqref{eq - 1ex}}{=}\frac{\tilde{\lambda}_\beta - 1-e^{\beta \eta}}{2(\tilde{\lambda}_\beta - 1)-e^{\beta\eta}}\]
	\[\stackrel{\eqref{eq - 0ex}}{=} \frac{\frac{-e^{\beta \eta} +\sqrt{ e^{2\beta \eta} + 4e^{2\beta \gamma}}}{2}}{\sqrt{ e^{2\beta \eta} + 4e^{2\beta \gamma}}}=\frac{1}{2}-\frac{e^{\beta \eta}}{2\sqrt{ e^{2\beta \eta} + 4e^{2\beta \gamma}}}.\]
	As $\eta>\gamma$ we conclude that  $\lim_{\beta\to+\infty}p_0^\beta=\frac{1}{2}-\frac{1}{2}=0$ and so
	\[\lim_{\beta \to \infty}\mu_{\beta A + B_\beta} = \delta_{1^{\infty}}. \]

\end{example}

\begin{remark}	It is possible to get similar conclusions as in above example in the case $B_\beta=B$ does not depend of $\beta$ and satisfies
\[ B(0,0) = B(1,0) = B(0,1)=0 \,\,\,and \,\,\, B(1,1)>0. \]
Indeed, just write $B(1,1)= \log (1+e^{\eta})$ where $\eta \in\mathbb{R}$ and replace $e^{\beta \eta}$ by $e^{\eta}$ in above computations in order to get $\lim_{\beta\to\infty} \frac{1}{\beta}\log( H_{\beta A + B_\beta}(1) )=-\gamma>0$ and $\lim_{\beta \to \infty}\mu_{\beta A + B} = \delta_{1^{\infty}}.$ 
\end{remark}

}

\end{document}